\documentclass{siamart220329}

\usepackage[utf8]{inputenc}
\usepackage[T1]{fontenc}
\usepackage{fullpage}
\usepackage{amsmath}
\usepackage{amssymb}
\usepackage{dsfont}
\usepackage{extarrows}
\usepackage{enumitem}
\usepackage{bm}
\usepackage{hyperref}
\hypersetup{colorlinks=true,
            linkcolor=blue,
            citecolor=blue,
            urlcolor=blue,
            pdftitle={mixture wasserstein barycenters}}
\usepackage{tikz}
\usepackage{pgfplots}
\usepackage{algorithm}
\usepackage{algorithmic}

\pgfplotsset{compat=1.16}
\graphicspath{ {.} }

\newcommand{\RR}{\mathbb{R}}
\newcommand{\NN}{\mathbb{N}}
\newcommand{\dd}[1][x]{\mathrm{d}#1}

\newcommand{\pos}{\mathbf{r}}
\newcommand{\charges}{\mathbf{z}}
\newcommand{\sol}{u_{\pos, \charges}}
\newcommand{\Esol}{E_{\pos, \charges}}
\newcommand{\zetasol}{\zeta_{\pos, \charges}}
\newcommand{\pisols}{\bm \pi^{\pos, \charges}}
\newcommand{\pisol}{\pi^{\pos, \charges}}
\renewcommand{\S}[2]{\mathrm{S}_{#1, #2}}
\newcommand{\M}{\mathcal{M}}
\newcommand{\T}{\mathcal{T}}
\newcommand{\icdf}{\mathrm{icdf}}
\newcommand{\cdf}{\mathrm{cdf}}

\newcommand{\dinfn}[1][n]{\mathrm{d}_{\infty, {#1}}}
\newcommand{\ddn}[1][n]{\mathrm{d}_{2, {#1}}}
\newcommand{\dtildeinfn}[1][n]{\tilde{\mathrm{d}}_{\infty, {#1}}}
\newcommand{\dtildedn}[1][n]{\tilde{\mathrm{d}}_{2, {#1}}}
\newcommand{\dn}[1][n]{\mathrm{d}_{2, {#1}}}
\newcommand{\pis}{\bm \pi}
\newcommand{\zetas}{{\bm \zeta}}

\newcommand{\distW}{W_{2}}
\newcommand{\distMW}{\mathrm{MW}_{2}}
\newcommand{\kbar}{\mathbf{k}}
\newcommand{\lbar}{\mathbf{l}}
\newcommand{\wstar}{w^*}
\newcommand{\barW}{\mathrm{Bar}^{\bm \lambda}_{\distW}}
\newcommand{\barMW}{\mathrm{Bar}^{\bm \lambda}_{\distMW}}
\newcommand{\appbarMW}{\overline{\mathrm{Bar}}^{\bm \lambda}_{\distMW}}
\newcommand{\bary}{\mathrm{Bar}^{\bm \lambda}}
\newcommand{\ssol}{u_{r}}
\newcommand{\szeta}{\zeta_{r}}
\newcommand{\msol}{u_{r}}
\newcommand{\SM}{\mathrm{SM}}

\DeclareMathOperator{\asinh}{arcsinh}

\DeclareMathOperator*{\argmin}{argmin}
\DeclareMathOperator*{\argmax}{argmax}

\newcommand{\zetabar}{{\zeta^{\bm\lambda}}}
\newcommand{\rbar}[1][\kbar]{{r^{\bm\lambda}_{#1}}}

\renewcommand{\L}{\mathrm{L}}
\newcommand{\C}{\mathcal{C}}
\renewcommand{\P}{\mathcal{P}}
\newcommand{\dP}{\mathrm{P}}

\DeclareMathOperator{\vect}{span}

\usepackage{stmaryrd}

\newtheorem{remark}[theorem]{Remark}

\numberwithin{equation}{section}
\numberwithin{theorem}{section}

\title{Nonlinear reduced basis using mixture Wasserstein barycenters:\\ application to an eigenvalue problem inspired from quantum chemistry}
\author{Maxime Dalery\thanks{Université Marie et Louis Pasteur, CNRS, LmB (UMR 6623), F-25000 Besançon, France (genevieve.dusson@math.cnrs.fr)} \and
        Geneviève Dusson\footnotemark[1] \and
        Virginie Ehrlacher\thanks{CERMICS, \'Ecole des Ponts \& Inria Paris, Paris, France} \and
        Alexei Lozinski\footnotemark[1]}
\date{}

\begin{document}
\maketitle

\begin{abstract}
The aim of this article is to propose a new reduced order modelling approach for parametric eigenvalue problems arising in electronic structure calculations. Namely, we develop
nonlinear reduced basis techniques for the approximation of parametric eigenvalue problems inspired from quantum chemistry applications. More precisely, we consider here a one-dimensional model which is a toy model for the computation of the electronic ground state wavefunction of a system of electrons within a molecule, solution to the many-body electronic Schr\"odinger equation, where the varying parameters are the positions of the nuclei in the molecule. We estimate the decay rate of the Kolmogorov $n$-width of the set of solutions for this parametric problem in several settings, including the standard $L^2$-norm as well as with distances based on optimal transport. The fact that the latter decays much faster than in the traditional $L^2$-norm setting motivates us to propose a practical nonlinear reduced basis method, which is based on an offline greedy algorithm, and an efficient stochastic energy minimization in the online phase.
We finally provide numerical results illustrating the capabilities of the method and good approximation properties, both in the offline and the online phase.
\end{abstract}



\section{Introduction}

In many academic and industrial applications, model order reduction techniques are used to accelerate the computation of the solutions to parametric partial differential equations. Techniques such as the reduced basis method give outstanding results for a large class of problems, see~\cite{Hesthaven2016-xm,Quarteroni2015-db}. 
A key factor determining the success of the method
is the approximability of the solution by a linear combination in a fixed vector space, possibly spanned by solutions for specific values of the parameters called snapshots. This ability is characterized by a fast-decreasing
of the
so-called Kolmorogov $n$-width, as described below. Such approach works very well in numerous cases, such as for linear elasticity equations~\cite{Milani2008-ne}, thermal equations~\cite{Rozza2010-ue}, see also~\cite{Quarteroni2011-rm} and references therein.
However, as was pointed out in~\cite{Ehrlacher2020-ef}, 
this approach proves ineffective in several cases,
especially when the solutions exhibit some transport of mass over parameter or time variation. This is, for instance, the case of the pure transport equation~\cite{Ohlberger2015-hu}. As the solution is translated over time, it is ineffective to approximate the solution as a linear combination of previous time steps, which would be the standard approximation using a linear reduced basis method. Another illustrative example is the electronic structure problem -- an eigenvalue problem where the solution tends to localize around the nuclei. This will be the problem of interest in this article. 
Also, again in~\cite{Ehrlacher2020-ef}, it is shown that for a simple Burgers equation, the Kolmogorov $n$-width decreases faster if one uses not a linear combination of previous snapshots, but Wasserstein barycenters between solutions, i.e. if there is some nonlinear transformation involved.

Recently, several works have built on this idea to propose nonlinear interpolations between solutions based on optimal transport. This is for example the case in~\cite{Iollo2022-ob}, where the authors propose a method based on an affine transformation of the snapshots to construct new approximations, or in~\cite{Nonino2019-ln}, where a preprocessing step using optimal transport is added to the offline phase.
In~\cite{Do2023-do}, a new method based on sparse Wasserstein barycenters is proposed. 
Other works include some machine-learning techniques to construct the nonlinear map, see e.g.~\cite{Romor2023-zv}, or to reconstruct higher frequency modes of the linear reduced basis operator from the low frequency modes, using trees or random forests~\cite{Cohen2023-xr}. 

The main limitation to the works based on optimal transport is the computational cost of Wasserstein barycenters, which do not scale well with the space dimension and the number of snapshots involved in the computation of the barycenter (solved through the so-called multi-marginal problem). A recent article~\cite{Delon2020-wk} has proposed a modified Wasserstein distance between mixtures of Gaussians, which was extended for more general mixtures in~\cite{Dusson2023-ah}.
This is particularly interesting for the electronic structure problem, where the solutions are often represented by a small number of functions of the same type, typically Gaussians or Slater functions. Indeed, this modified Wasserstein distance 
enables the computation of barycenters without suffering from the curse of dimensionality,
since the problem dimension depends on the number of components in the mixtures, and not on a potentially large spatial grid parameter.

In this article, we fully use this mixture distance to propose a nonlinear reduced basis method based on optimal transport. As is standard in reduced basis methods, our algorithm works in two stages. In the offline stage, a collection of snapshots is gathered using a greedy algorithm, selecting,  at each step, the worse-approximated snapshot in 
the set of barycenters between previously selected snapshots for this particular mixture distance.
In the online stage, i.e. when one wants to compute the solution for a new parameter, we minimize the energy on the set of barycenters between selected snapshots. It is a nonlinear problem but in a low-dimensional parameter space, so that the online cost stays reasonable.

The outline of this article is as follows. In Section~\ref{sec:prelim}, we present the settings, namely we present the one-dimensional eigenvalue toy problem we mainly focus on in this work, and we detail a few preliminaries on the Kolmogorov $n$-width, as well as on optimal transport and Wasserstein barycenters.
In Section~\ref{sec:kolmogorov} we state our main theoretical results about the estimations for the Kolmogorov $n$-width of the set of solutions of the one-dimensional toy problem, in different settings: 
with a linear approximation, with a Wasserstein transport metric, as well as a Wasserstein-type metric for mixtures. 
Our results are similar in spirit to those proved of~\cite{Ehrlacher2020-ef}, with the specificity that we focus on solution sets stemming from parametrized electronic structure calculation problems and that we consider the mixture Wasserstein metric~\cite{Delon2020-wk,Dusson2023-ah}, in addition to the exact Wasserstein and $L^2$ metrics. The proofs of the theoretical results are postponed to Section~\ref{sec:proofs}.
In Section~\ref{sec:nonlinearRB}, we present the nonlinear reduced basis method proposed in this work, which consists of an offline and an online stage. 
Let us emphasize that, despite the fact that we only illustrate its numerical behaviour on the one-dimensional toy problem mentioned above, the proposed numerical strategy  can in principle be adapted to solve realistic electronic structure problems in any dimension. 
This is the focus of ongoing work. The main novelty of the approach, in particular compared to the one proposed in~\cite{Ehrlacher2020-ef} is that, in the online phase, we use an energetic variational principle to determine the solution of the nonlinear reduced order model. Finally, we present numerical results in Section~\ref{sec:num}.

\section{Preliminaries} \label{sec:prelim}

The aim of this section is to introduce some preliminaries. We first present in Section~\ref{sec:modelproblem} the considered parametric eigenvalue problem, which is motivated by quantum chemistry applications. We then provide some definitions about Kolmogorov widths in Section~\ref{sec:Kolmogorov}, and we recall some fundamentals about the Wasserstein and Mixture Wasserstein metrics in Section~\ref{sec:OT}.

\subsection{An eigenvalue problem inspired from quantum chemistry}\label{sec:modelproblem}

In this article, we focus on the following one-dimensional eigenvalue partial differential
equation parameterized by $\pos := \left(r_1, \dots, r_M\right)\in\RR^M$
and $\charges := \left(z_1,\dots, z_M\right)\in(\RR^*_+)^M$ for $M\in \NN^*$.
More precisely, we are looking for the lowest eigenvalue $\Esol \in \RR$ and a corresponding eigenstate $\sol \in H^1(\RR)$  satisfying 
\begin{equation}\label{prob:schr_main}
        \displaystyle - \frac{1}{2}\sol'' + \left(-\sum_{m=1}^M z_m\delta_{r_m}\right)\sol= \Esol\sol.
\end{equation}
This problem can be seen as a toy ground state electronic structure problem, with an Hamiltonian of the form $-\frac12\Delta + V$, with a potential $V$ taken as a sum of Dirac masses
$V:=-\sum_{m=1}^M z_m\delta_{r_m}(x)$ localized at atomic positions $\pos$ with charges $\charges$.
In this simple framework, it can be easily checked that any eigenvector solution to this problem actually belongs to $L^1(\mathbb{R})$. More precisely, there exists a unique strictly positive
eigenvector solution to this problem such that $\| \sol\|_{L^1(\mathbb{R})} = 1$, and it is explicitly given by~\cite[Section 3.1]{Pham2017-bd}
\begin{equation} \label{eq:sol}
    \sol = \sum_{m=1}^M \pisol_m \S{\zetasol}{r_m},
\end{equation}
with $\pisols = \left(\pisol_m\right)_{m=1}^M \in (\RR_+)^M$ of total sum equal to $1$, $\zetasol >0$, and where for all $\zeta >0$ and all $r\in \mathbb{R}$, the Slater function $\S{\zeta}{r}$ is defined by 
\begin{equation} \label{eq:slater}
    \S{\zeta}{r}: x \longmapsto \frac{\zeta}{2} e^{-\zeta|x - r|}.
\end{equation}
Note that the normalization with respect to the $L^1$-norm in problem~\eqref{prob:schr_main} is not standard, but $u_{\pos,\charges}$ can thus be interpreted as the density associated with a probability measure on $\mathbb{R}$, which will be an essential feature in the following. In the rest of the article, we will make an abuse of notation and identify an absolutely continuous probability measure with its associated probability density.
We refer the reader to~\cite{Pham2017-bd} for an extensive review on the link between this toy one-dimensional model and actual electronic structure calculation problems in molecules.

The eigenvalue problem~\eqref{prob:schr_main} can be equivalently formulated as an energy minimization problem
\begin{equation}
    \label{eq:energy_min}
        \min_{\substack{u\in H^1(\RR) \\ \|u\|_{L^1(\RR)}=1}}
    \frac{\Esol(u)}{\|u\|_{L^2(\RR)}^2},
\end{equation}
with 
\[
    \Esol(u) := \frac12 \int_\RR |u'|^2 - \sum_{m=1}^M z_m u(r_m)^2.
\]

\medskip

Let us mention here some particular explicit formulas available for small values of $M$.
\begin{itemize}
\item[\bfseries Case~1:] For $M = 1$, $\pos = (r)$ and $\charges = (z)$ for some $r\in \mathbb{R}$ and $z>0$, it holds that (see ~\cite[Theorem~3.4]{Pham2017-bd})
\begin{equation} \label{eq:mono_zeta}
    \zetasol = z.
\end{equation}
\item[\bfseries Case~2:] For $M = 2$, $\pos = (-r, r)$ and $\charges = (z, z)$ with $r,z>0$, it holds that (see~\cite[Corollary~3.3]{Pham2017-bd})
\begin{equation} \label{eq:sym2_zetapis}
    \zetasol = z + \frac{W\left(2zre^{-2zr}\right)}{2r}
    \text{~~~and~~~}
    \pisol = \left( \frac{1}{2}, \frac{1}{2} \right),
\end{equation}
where $W$ is the Lambert function defined as the inverse of the function $x \mapsto xe^x$.
In this special case, we also have the following equality on $\zetasol$:
\begin{equation} \label{eq:zeta_eq}
    \zetasol = z \left( 1 + e^{-2\zetasol r} \right).
\end{equation}
\end{itemize}

\subsection{Kolmogorov widths in metric spaces}
\label{sec:Kolmogorov}

Let us now introduce some definitions, which will play the role of generalized Kolmogorov $n$-widths in our setting.
We first start by recalling the definition of the Kolmogorov $n$-width in a Hilbert space $\mathbb{H}$ endowed with a scalar product $\langle \cdot, \cdot \rangle$ and associated norm $\|\cdot\|$. We denote by $\dP_V:\mathbb{H} \longrightarrow V$ the projection operator onto a closed vector subspace $V\subset \mathbb{H}$. 

We also denote by $\mathbf{Z}\subset\RR^p$ a subset of parameter values, and for all $z\in \mathbf{Z}$, we assume that $u(z)$ is an element of $\mathbb{H}$. Finally, we denote by $\mathcal E$ the following set
$$
\mathcal E:=  \{u(z),~z\in\mathbf{Z}\} \subset \mathbb{H}. 
$$

\begin{definition}
The $\L^\infty$ Kolmogorov $n$-width of $\mathcal{E}$ is defined by
    \begin{align}
        \dinfn(\mathcal{E}, \mathbb{H})
            := \inf_{\substack{V_n\subset \mathbb{H}\\\dim V_n=n}}\sup_{z\in \mathbf{Z}}~\Vert u(z)-\dP_{V_n}u(z)\Vert. \label{eq:def_dn}
    \end{align}
    The $\L^2$ Kolmogorov $n$-width of $\mathcal E$ is defined by
    \begin{align}
        \ddn (\mathcal{E}, \mathbb{H}) :&= 
         \inf_{\substack{V_n\subset \mathbb{H}\\\dim V_n=n}} \left(\int_{z\in\mathbf{Z}}~\Vert u(z)-\dP_{V_n}u(z)\Vert^2\dd[z]\right)^{1/2}. \label{eq:def_deltan}
    \end{align}
\end{definition}

Let us point out that it can easily be checked that
\begin{equation}\label{eq:ineq}
 \ddn (\mathcal{E}, \mathbb{H})  \leq |\mathbf{Z}|^{1/2}\dinfn(\mathcal{E}, \mathbb{H})
\end{equation}
where $|{\mathbf{Z}}|$ refers to the Lebesgue measure of the set $\mathbf{Z}$. 

\medskip

In the following, we extend this definition in a meaningful way to the case where $\mathcal E$ is not a subset of a Hilbert space $\mathbb{H}$, but of a metric space $\mathbb{M}$ equipped with a distance $\delta$. 
First a natural generalization consists in replacing the quantity $\Vert u(z)-\dP_{V_n}u(z)\Vert$ by the quantity $\displaystyle \mathop{\inf}_{v_n \in V_n} \delta(u(z), v_n)$. Note however that there is no notion of vectorial subspace in a metric space.
Instead, we consider the notion of barycenters whenever this is well-defined. In a formal way, this corresponds for a given ${\bm u}:=( u_1, \ldots, u_n) \in \mathbb{M}^n$ to 
generating all possible barycenters defined as
\begin{equation} \label{eq:barycenter_metric_space}
     \bary_\delta({\bm u}) \in \mathop{\argmin}_{u\in \mathbb{M}} \sum_{i=1}^n\lambda_i \delta(u, u_i)^2, 
\end{equation}
for $\bm\lambda = (\lambda_1,\ldots,\lambda_n)\in \Lambda_n$ in a subset $\Lambda_n$ of $\mathbb{R}^n$ guaranteeing that problem~\eqref{eq:barycenter_metric_space} is well-posed.
We assume in the following that the metric~$\delta$ is such that there exists at least one solution $ \bary_\delta({\bm u})$ to problem \eqref{eq:barycenter_metric_space}. 
In the case when $\mathbb{M}$ is a Hilbert space $\mathbb{H}$ and the metric $\delta$ is defined by $\delta(u,v) =\|u-v\|$ for all $u,v\in \mathbb{H}$, 
the solution to problem~\eqref{eq:barycenter_metric_space} is unique if and only if $\sum_{i=1}^n \lambda_i >0$ and is given by $\sum_{i=1}^n \lambda_i u_i$ provided $\sum_{i=1}^n \lambda_i =1$.
In general, we will denote in the sequel $\mathcal B_\delta^{\bm \lambda}({\bm u})$ the set of minimizers to~\eqref{eq:barycenter_metric_space}.

\medskip

Assume now that $\mathcal E \subset \mathbb{M}$. The most straightforward extension of the notion of Kolmogorov $n$-width in a metric space setting is given in the following definition.
\begin{definition}
    \label{def:kolmo_nwidth}
    The metric $\L^\infty$ Kolmogorov $n$-width of the set $\mathcal{E} \subset \mathbb{M}$ is defined by
    \begin{equation*}
        \dinfn(\mathcal{E}, \mathbb{M}) = \inf_{{\bm m}\in \mathbb{M}^n}
            \sup_{z\in\mathbf{Z}}~\inf_{\bm\lambda\in \Lambda_n} ~\inf_{b \in \mathcal B_\delta^{\bm \lambda}({\bm m})} \delta (u(z), b).
    \end{equation*}
    Similarly, the metric $\L^2$ Kolmogorov $n$-width is defined by
    \begin{equation*}
        \ddn(\mathcal{E}, \mathbb{M}) := \inf_{{\bm m}\in \mathbb{M}^n}
            \left(\int_{z\in\mathbf{Z}}~ ~\inf_{\bm\lambda\in \Lambda_n} ~\inf_{b \in \mathcal B_\delta^{\bm \lambda}({\bm m})} \delta (u(z), b)^2 \right)^{1/2}.
    \end{equation*}
\end{definition}

\subsection{Wasserstein spaces}\label{sec:OT}

The aim of this section is to present some preliminairies about Wasserstein and mixture Wasserstein spaces. 

\subsubsection{Wasserstein metric}

Let $\P_2(\RR)$ denote the set of probability measures on $\RR$ with finite second-order moments. The 2-Wasserstein distance over $\P_2(\RR)$ is defined for $u, v \in \P_2(\RR)$ as
\begin{equation*}
    \distW(u, v) := \inf_{\pi \in \Pi(u, v)} \left( \int_{\RR^2} (x - y)^2\dd[\pi(x, y)] \right)^{1/2},
\end{equation*}
where $\Pi(u, v)$ is the set of probability measures over $\RR^2$ with marginals $u$ and $v$, which is called the set of transport plans between $u$ and $v$. Note that $(\mathcal P_2(\RR),W_2)$ is a geodesic metric space~\cite{Delon2020-wk}.

For any measurable map $T: \mathbb{R} \to \mathbb{R}$ and any probabilty measure $\rho\in \mathcal P_2(\mathbb{R})$, the push-forward measure $T\#\rho$ is defined as the probability measure on $\mathbb{R}$ defined so that for all measurable sets $B\subset \mathbb{R}$, $T\#\rho(B) = \rho(T^{-1}(B))$.

For $n\in \mathbb{N}^*$, let us denote by 
\[
\Lambda_n = \left\{ 
(\lambda_1,\ldots,\lambda_n) \in (\RR_+)^n, \quad \sum_{i=1}^n{\lambda_i} = 1
\right\}
\]
the set of barycentric weights of cardinality $n$. The Wasserstein barycenter of a collection of $n$ probability measures ${\bm u}:=(u_1, \dots, u_n) \in \P_2(\RR)^n$ associated to a set of barycentric weights ${\bm \lambda} := (\lambda_1,\ldots,\lambda_n) \in \Lambda_n$  is then defined (see~\cite{Agueh2011-uz}) as the unique solution to the problem
\begin{equation}\label{eq:barycenter}
     \mathop{\inf}_{u\in \P_2(\RR)} \sum_{i=1}^n \lambda_i \distW(u, u_i)^2.
\end{equation}
The unique minimizer of~\eqref{eq:barycenter} is  denoted by $\barW({\bm u})$.
This barycenter is also related to the so-called multi-marginal optimal transport problem~\cite{Agueh2011-uz,Gangbo1998-wi}, defined, given $n$ elements
$\bm u = (u_1, \dots, u_n)$ in $\P_2(\RR)^n$, as
\begin{equation}\label{eq:WassMulti}
        m\distW({\bm u}; {\bm \lambda}) := \inf_{\pi \in \Pi(u_1, \dots, u_n)} \left( \int_{\RR^n}
        \frac{1}{2} \sum_{i, j = 1}^n \lambda_i \lambda_j  (x_i - x_j)^2 \dd[\pi(x_1,\ldots,x_n)] \right)^{1/2},
\end{equation}
where $\Pi(u_1, \dots, u_n)$ is the set of probability measures over $\RR^n$ with marginals $u_1, \dots, u_n$, and there holds
\begin{equation*}
    m\distW({\bm u}; {\bm \lambda}) =  \left[ \sum_{i=1}^n \lambda_i \distW(\barW({\bm u}), u_i)^2 \right]^{1/2}.
\end{equation*}

In the present one-dimensional setting, the Wasserstein distance and barycenter can be expressed in a more direct way using the inverse cumulative distribution function of the considered probability measures. More precisely, 
we introduce the cumulative distribution function ($\cdf$) of an element $u \in \P_2(\RR)$ as
\[
\cdf_u : x \in \RR \mapsto \displaystyle \int_{-\infty}^x \dd[u](t)
\]
and its inverse cumulative distribution function ($\icdf$) as the generalized inverse of the $\cdf$:
\begin{equation*}
    \icdf_u :
    \left\{ \begin{array}{ccc}
        [0, 1] & \longrightarrow & \RR, \\
        s & \longmapsto & \cdf_u^{-1} := \inf\{x \in \RR,~ \cdf_u(x) > s \}.
    \end{array} \right.
\end{equation*}
Then, for any $(u, v) \in \P_2(\RR)^2$, there holds
\begin{equation} \label{eq:W2_icdf}
    \distW(u, v) = \| \icdf_u - \icdf_v \|_{\L^2([0,1])},
\end{equation}
and for any set of barycentric weights ${\bm \lambda}:= (\lambda_1, \ldots, \lambda_n)\in \Lambda_n$  and ${\bm u}:= (u_1, \dots, u_n) \in \P_2(\RR)^n$, the icdf of the barycenter $\barW({\bm u})$ satisfies
\begin{equation}\label{eq:icdfbary}
    \icdf_{\barW({\bm u})} = \sum_{i = 1}^n \lambda_i \; \icdf_{u_i}.
\end{equation}
We will significantly use this convenient property \eqref{eq:icdfbary}, which is specific to the one-dimensional setting, in our analysis. 

\medskip

\subsubsection{Slater mixture Wasserstein metric}

In higher dimensional spaces, such a characterization does not exist, so that the computation of Wasserstein distances and barycenters is more involved. However, for some specific classes of probability distributions, the Wasserstein distances and barycenters are explicit. This is the case for Gaussian distributions, and more generally, for all location-scatter distributions~\cite{Alvarez-Esteban2016-us}, i.e. all distributions that can be related with an affine transportation map, see also~\cite[Section~4.1]{Dusson2023-ah}. 
In this contribution, we will draw a particular interest to Slater distributions, as defined in~\eqref{eq:slater}.
Noting that the mean of the Slater distribution $\S{\zeta}{r}$ is $r$ and the variance is $2/\zeta^2$,
we can easily obtain the explicit expression of the Wasserstein distance between two Slater distributions from~\cite[Theorem~2.3]{Alvarez-Esteban2016-us}. More precisely, for $\zeta_1, \zeta_2>0$ and $r_1,r_2\in\RR$, there holds
\begin{equation}\label{eq:WassSlater}
        \distW\left(\S{\zeta_1}{r_1}, \S{\zeta_2}{r_2}\right)^2 = (r_1 - r_2)^2 + 2\left( \frac{1}{\zeta_1} - \frac{1}{\zeta_2} \right)^2.
\end{equation}

Moreover, thanks to~\cite[Theorem 2.4]{Alvarez-Esteban2016-us}, it is also possible to obtain an explicit expression for the barycenter between $n\in \mathbb{N}^*$ Slater distributions. 
Indeed, for ${\bm \lambda}:=(\lambda_1, \ldots, \lambda_n)\in \Lambda_n$, ${\bm \zeta}:=(\zeta_1, \ldots, \zeta_n)\in (\mathbb{R}_+^*)^n$, ${\bm r}:=(r_1, \ldots, r_n)\in \mathbb{R}^n$, denoting by ${\bm S}_{{\bm \zeta},{\bm r}}:= (\S{\zeta_1}{r_1}, \dots,  \S{\zeta_n}{r_n})$, 
\begin{equation}\label{eq:barySlater}
        \barW\left({\bm S}_{{\bm \zeta},{\bm r}}\right) = \S{\zeta^{\bm \lambda}}{r^{ {\bm \lambda}}},
\end{equation}
    with
\begin{equation}\label{eq:formula}
        \zeta^{\bm \lambda} = \left[\sum_{i=1}^n \frac{\lambda_i}{\zeta_i} \right]^{-1}
            \text{~~~and~~~}
            r^{{\bm \lambda}} = \sum_{i=1}^n \lambda_i r_i.
\end{equation}

Since the solutions of problem~\eqref{prob:schr_main} are convex combinations of Slater distributions, as stated in~\eqref{eq:sol}, we are interested in approaches to efficiently compute Wasserstein-like distances and barycenters for such objects. To this aim, let us start by precisely defining mixtures of Slater distributions. 
A mixture of Slater distributions $m$ is a finite convex combination of Slater distributions i.e. a probability distribution such that
there exists $K \in \NN^*$, a $K$-tuple $\zetas = (\zeta_1, \dots, \zeta_K)\in (\RR_+^*)^K$,
a position tuple $\pos = (r_1, \dots, r_K)\in\RR^K$ and barycentric weights $\pis = (\pi_1, \dots, \pi_K)\in \Lambda_K$
such that
\begin{equation*}
    m = \sum_{k = 1}^K \pi_k \S{\zeta_k}{r_k}.
\end{equation*}
Such a probability distribution will be denoted by $m_{{\bm \zeta}, {\bm r}, {\bm \pi}}$ in the following.
Therefore, the solution of~\eqref{prob:schr_main} with parameters $M$, $\pos$, and $\charges$ is a mixture of Slater distributions with $K = M$, position parameters $\pos$,
scale parameters $\zeta_1 = \dots = \zeta_K = \zetasol$ and weights $\pis = \pisols$.

\medskip

We denote by  $\SM(\RR)$ the subset of $\P_2(\RR)$ of Slater mixtures, that is
\begin{align}
\label{eq:SMR}
      \SM(\RR) := \Big\{  &
    m \in \P_2(\RR), \; \exists K\in\NN^*, \; 
     \exists \zeta_1, \dots, \zeta_K\in \RR^*_+, \;
    \exists r_1,\ldots, r_K \in \RR
    , \; \\&
    \exists \pis = (\pi_1, \dots, \pi_K)\in \Lambda_K, \;
    m = \sum_{k = 1}^K \pi_k \S{\zeta_k}{r_k}
    \Big\}. \nonumber
\end{align} 
In order to compute meaningful and easily computable distances, we in fact endow this space $\SM(\RR)$ not with the Wasserstein distance, but with a modified Wasserstein distance, as was proposed in~\cite{Delon2020-wk,Chen2018-vn} for mixtures of Gaussians, and recently extended to the case of Slater functions in~\cite{Dusson2023-ah}. 
Therefore, we endow $\SM(\RR)$ with this Wasserstein-type distance denoted by $\distMW$. More precisely, for $m^1, m^2\in\SM(\RR)$, 
with parameters $K^1,K^2\in\NN^*$, ${\bm \zeta}^1:=(\zeta^1_1, \ldots, \zeta^1_{K_1})\in (\RR_+^*)^{K_1}$, ${\bm \zeta}^2=(\zeta^2_1, \ldots, \zeta^2_{K_2})\in (\RR_+^*)^{K_2}$, 
$ {\bm r}^1:=(r^1_1,\ldots,r^1_{K_1})\in \RR^{K_1}$, ${\bm r}^2:=(r^2_1, \ldots, r^2_{K_2})\in \RR^{K_2}$, ${\bm \pi}^1=(\pi_1^1, \ldots, \pi_{K_1}^1) \in \Lambda_{K^1},$
${\bm \pi}^2=(\pi_1^2, \ldots, \pi_{K_2}^2) \in \Lambda_{K^2},$
namely
\[
    m^1 = m_{{\bm \zeta}^1, {\bm r}^1, {\bm \pi}^1} = \sum_{k^1 = 1}^{K^1} \pi^1_{k^1} m^1_{k^1}, \quad  
    m^2 = m_{{\bm \zeta}^2, {\bm r}^2, {\bm \pi}^2}= \sum_{k^2 = 1}^{K^2} \pi^2_{k^2} m^2_{k^2},
\]
where
\[
    \text{for }i = 1,2, \quad \text{for } k^i = 1,\ldots, K^i, \quad m^i_{k^i} = \S{\zeta^i_{k^i}}{r^i_{k^i}},
\]
then the modified Wasserstein distance on the set of Slater mixtures is defined through the following minimization problem
\begin{equation}
    \label{eq:mixture_distance}
    \distMW(m^1, m^2) := \left[ \min_{w \in \Pi(\pis^1, \pis^2)} \sum_{k^1 = 1}^{K^1} \sum_{k^2 = 1}^{K^2}
    w_{k^1, k^2} \distW^2\left(m^1_{k^1}, m^2_{k^2}\right) \right]^{1/2},
\end{equation}
where
\begin{equation*}
    \Pi(\pis^1, \pis^2) = \left\{ w \in \RR_+^{K^1 \times K^2}, ~ \sum_{k^1 = 1}^{K^1} w_{k^1, k^2} = \pi^2_{k^2}, ~ \sum_{k^2 = 1}^{K^2} w_{k^1, k^2} = \pi^1_{k^1} \right\}.
\end{equation*}
The minimization problem~\eqref{eq:mixture_distance} is well-posed as the function to minimize is continuous on the bounded closed finite dimensional set $\Pi(\pis^1, \pis^2)$. 
Note that the Wasserstein distances $W_2^2(m_{k^1}^1, m_{k_2}^2)$ can be analytically computed using formula (\ref{eq:WassSlater}). In~\cite{Dusson2023-ah}, it is proved that $(\SM(\RR),MW_2)$ is a geodesic metric space.

A multi-marginal optimal transport problem can similarly be defined for the distance $\distMW$.
More precisely, for all ${\bm \lambda} = (\lambda_1, \ldots, \lambda_n)\in \Lambda_n$ and all ${\bm m} = (m^1, \ldots, m^n)\in {\rm SM}(\mathbb{R})^n$, we introduce
\begin{equation}
\label{eq:mMW2}
   m\distMW({\bm m}; {\bm \lambda} ) := \left[ \min_{w \in \Pi(\pis^1, \dots, \pis^n)} \sum_{\kbar\in\mathbf{K}}
    w_{\kbar} \; m\distW^2\left(m^1_{k^1}, \dots, m^n_{k^n}; {\bm \lambda}\right) \right]^{1/2},
\end{equation}
where $\kbar = (k^1, \dots, k^n)$ is in $\mathbf{K} = \{1, \dots, K^1\}\times\dots\times\{1, \dots, K^n\}$ with
\[
\forall i=1,\ldots n, \quad m^i = \sum_{k^i = 1}^{K^i} \pi^i_{k^i} m^i_{k^i},
\]
so that $(w_{\kbar})_{\kbar \in\mathbf{K}}$ is an $n$-order tensor with respective dimensions $(K^1,\ldots, K^n)$. Recall that the quantities $m\distW^2\left(m^1_{k^1}, \dots, m^n_{k^n}; {\bm \lambda}\right)$ are defined using formula~\eqref{eq:WassMulti}.
Moreover, $\Pi(\pis^1,\ldots \pis^n)$ is the set of such tensors with non-negative coefficients satisfying the constraints
\[
    \forall i=1,\ldots, n, \quad \sum_{k^1,\ldots, k^{i-1}, k^{i+1}, \ldots, k^n } w_{k^1,\ldots, k^n} = \pi^i_{k^i},
\]
where $\pis^i = (\pi^i_{k^i})_{1\leq k^i\leq K^i}$. 

\medskip

A barycenter between $n$ Slater mixtures ${\bm m} = (m^1, \ldots, m^n)\in SM(\mathbb{R})^n$ with barycentric weights ${\bm \lambda} = (\lambda_1, \ldots, \lambda_n)\in \Lambda_n$ for the mixture distance $\distMW$ (see \eqref{eq:barycenter_metric_space}) is defined as an element $\barMW\left({\bm m}\right) \in {\rm SM}(\mathbb{R})$ solution to the following minimization problem: 
\begin{equation}\label{eq:barymixture}
\barMW\left({\bm m}\right) \in \mathop{\rm argmin}_{b\in {\rm SM}(\mathbb{R})} \sum_{i=1}^n \lambda_i \; {\rm MW}_2(b,m^i)^2.
\end{equation}
 
It is shown in~\cite{Dusson2023-ah} that there always exists at least one solution to \eqref{eq:barymixture} and that any solution can be expressed as 
 \begin{equation}
        \label{eq:bar_mixtures}
        \barMW\left({\bm m}\right) =
        \sum_{\kbar \in \mathbf{K}} (w^*({\bm m}; {\bm \lambda}))_{\kbar} \barW\left(m^1_{k^1}, \dots, m^n_{k^n}\right),
\end{equation}
where $w^*({\bm m}; {\bm \lambda})$ is a solution to the minimization problem~\eqref{eq:mMW2}.

\begin{remark}
    \label{rem:sparsity}
    Note that all minimizers $w^*({\bm m};{\bm \lambda})$ of problem~\eqref{eq:mMW2} have at most $K_1+K_2+\ldots+K_n-n+1$ nonzero components, see e.g.~\cite{Delon2020-wk}.
\end{remark}

\section{Theoretical estimations of the decay of Kolmogorov \texorpdfstring{$n$}{n}-widths} \label{sec:kolmogorov}

In this section, we provide estimates for metric Kolmogorov $n$-width of sets of solutions of the one-dimensional toy model presented in Section~\ref{sec:modelproblem} using three different metrics. We state our main results here and postpone their proofs to Section~\ref{sec:proofs}. First, we study the Kolmogorov $n$-width with an underlying $L^2$-norm, which is the standard choice in linear reduced order modelling. We then turn to the  metric Kolmogorov $n$-width with respect to the Wasserstein metric, which in the one-dimensional case, can be recast as a Kolmogorov $n$-width for the $L^2$-norm on the inverse cumulative distribution functions of the solutions, and we show that the decay rate is faster than for the traditional $L^2$-norm. Finally, we show that this decay can be improved by considering the modified Wasserstein metric defined in~\eqref{eq:mixture_distance}.

\subsection{On linear approximations: case \texorpdfstring{$\mathbb{H} = L^2(\mathbb{R})$}{H = L2(R)}}

In this section, we state estimates on the $L^2$ Kolmogorov $n$-width decay of a set of solutions of~\eqref{prob:schr_main}
where the parameters vary in a compact set.
We consider the charges $\charges$ to be fixed and define the set of solutions
\begin{equation*}
    \M_\charges^R := \left\{ \sol,~ \pos\in[-R, R]^M \right\},
\end{equation*}
where $R \in \RR_+$ is a given parameter and $\sol$ is the solution to~\eqref{prob:schr_main}. We first consider the case $M = 1$, i.e. the case where the potential is a single Dirac delta.
We recall in~\eqref{eq:mono_zeta} the particular form of $\zetasol$
and denote by $\msol = \sol$ the solution of~\eqref{prob:schr_main}.
We also denote $\M_z^R$ the above set in this case.

\begin{theorem} \label{thrm_linear_comp}
    There exist positive constants $c_R$, $C_R$, $\tilde{c}_R$ and $\tilde{C}_R$ depending on $R$
    such that for all $n\in\NN^*$,
    \begin{equation} \label{eq:thrm:linear_comp:dn}
        c_R n^{-\frac{3}{2}} \leqslant \dinfn(\M_z^R, \L^2(\RR)) \leqslant C_R n^{-\frac{3}{2}},
    \end{equation}
    and
    \begin{equation} \label{eq:thrm:linear_comp:deltan}
        \tilde{c}_R n^{-\frac{3}{2}} \leqslant \ddn(\M_z^R, \L^2(\RR)) \leqslant \tilde{C}_R n^{-\frac{3}{2}}.
    \end{equation}
\end{theorem}

The proof of Theorem~\ref{thrm_linear_comp} is presented in Section~\ref{sec:proof_thrm_linear_comp}. Next, we claim that a similar result holds true for any system of several fixed charges $\charges$.
\begin{corollary} \label{coro_linear_comp}
    The following lower bound 
    \begin{equation} \label{eq:coro:linear_comp:dn}
        c_R n^{-\frac{3}{2}} \leqslant \dinfn(\M_\charges^R, \L^2(\RR)) 
    \end{equation}
   holds with the positive constant $c_R$ from (\ref{eq:thrm:linear_comp:dn}).
\end{corollary}

\begin{proof}
    Remark that for $\pos = (r, \dots, r)$, the problem~\eqref{prob:schr_main}
    corresponds to $M = 1$,
    with a unique charge $\displaystyle z = \sum_{m = 1}^M z_m$. 
    Thus, $\dinfn(\M_z^R, \L^2(\RR))\leqslant \dinfn(\M_\charges^R, \L^2(\RR))$ and (\ref{eq:coro:linear_comp:dn}) follows from (\ref{eq:thrm:linear_comp:dn}).
\end{proof}

\begin{remark}
    It does not seem trivial to obtain a similar upper bound as the lower bound in~\eqref{eq:coro:linear_comp:dn}, nor the bounds for $\dn(\M_\charges^R, \L^2(\RR))$. It is not clear either whether the bound~\eqref{eq:coro:linear_comp:dn} is optimal. 
    However, since our point is to show that the Kolmogorov $n$-width of the set of solutions in $L^2(\mathbb{R})$ is larger than the generalized Kolmogorov $n$-width using Wasserstein metric, the lower bound (even suboptimal) is sufficient here.
\end{remark}

\subsection{On the Wasserstein transport metric: case \texorpdfstring{$\mathbb{M} = \mathcal P_2(\mathbb{R})$}{M = P2(R)} with \texorpdfstring{$\delta = W_2$}{d = W2}}\label{sec:32}

Let us now give some insight on what happens when we consider the Wasserstein distance $\delta = W_2$ with $\mathbb{M} =\mathcal P_2(\mathbb{R})$ and comment on the theoretical results.

\subsubsection{Preliminary discussion}

It holds from~\eqref{eq:W2_icdf} that barycenters can be expressed using the icdf of $m_1,\ldots, m_n$. 
there holds for any $m\in\mathcal P_2(\RR)$ and ${\bm m} = (m_1,\ldots,m_n)\in\mathcal P_2(\RR)^n$
\[
    W_2(m, \bary_{W^2}({\bm m}))
    = \left\| \icdf_{m} - \sum_{i=1}^n \lambda_i \icdf_{m_i} \right\|_{\L^2([0,1])}.
\]

It is then natural to relate the generalized Kolmogorov $n$-width of a set $\mathcal E \subset \mathcal P_2(\mathbb{R})$ associated with the Wasserstein metric with the Kolmogorov $n$-width of the set $\T = \{ \icdf_{u(z)}, ~z\in\mathbf{Z} \}$ in the Hilbert space $L^2(0,1)$. It can then be easily checked that
\[
    \dinfn(\mathcal{E}, \P_2(\RR)) \geq \dinfn(\T, \L^2([0,1])),
\]
and similarly that
\[
    \ddn(\mathcal{E}, \P_2(\RR)) \geq \ddn(\T, \L^2([0,1])). 
\]
Proving decay estimates on the quantities $\dinfn(\mathcal{E}, \P_2(\RR))$ and $\ddn(\mathcal{E}, \P_2(\RR))$ appears to be a difficult task. In the present work, we manage to prove decay estimates with respect to $n$ of $\dinfn(\T, \L^2([0,1]))$ and $\ddn(\T, \L^2([0,1]))$. We first would like to point out that the latter quantities are also the ones for which decay estimates have been proven for some conservative transport equations in~\cite{Ehrlacher2020-ef}.

\subsubsection{Main results}

In this section, we give estimates of $L^2(0,1)$ Kolmogorov $n$-widths decays of the set of icdfs of solutions, and show that they converge faster to zero as a function of $n$ than the $L^2(\mathbb{R})$ Kolmogorov $n$-widths of the original solution set.
First, in the case $M = 1$ with a fixed charge $z$, we have the following proposition.
\begin{proposition}
    The $L^2(0,1)$ Kolmogorov $n$-width of the set of $\icdf$ of solutions $\left\{ \icdf_{\msol},~ r\in\RR \right\}$
    with an unbounded position parameter $r \in \RR$ is equal to zero for $n > 1$.
\end{proposition}

\begin{proof}
    The solutions $\msol$ are translations one to another, hence the set of $\icdf$
    of solutions is a subset of the $2$-dimensional vector space
    $\vect\{ \icdf_{u_0}, \mathds{1}_{(0, 1)} \}$.
\end{proof}

In the rest of Section~\ref{sec:32}, we consider symmetric systems with $M = 2$, i.e. $\pos = (-r, r)$ and $\charges = (z, z)$.
For simplicity of notation, we denote by $\ssol = \sol$ the ground state of the above system
and $\szeta = \zetasol$ which is given by~\eqref{eq:sym2_zetapis}.

\bigskip 

\paragraph{\bfseries Transport approximation for $r\leq R$}

We first consider the set of icdf of solutions where the position parameter $r$ is bounded by some positive $R>0$. Let us introduce the set 
\begin{equation*}
    \T_R^- := \{ \icdf_{\ssol},~ r\in[0,R] \} \subset L^2(0,1).
\end{equation*}
We show a better Kolmogorov $n$-width decay than its counterpart for the original solution set.

\begin{theorem} \label{thrm:transp_before}
	There exists a constant $C > 0$ independent of $R$ such that for all $n\in\mathbb{N}^*$,
	\begin{equation*}
		\dinfn(\T_R^-, \L^2(0, 1)) \leqslant Ce^{5\zeta_R R} n^{-\frac{5}{2}}.
	\end{equation*}
\end{theorem}

The proof of Theorem~\ref{thrm:transp_before} is postponed to Section~\ref{sec:proof_thrm_transp_before}.

\bigskip 
\paragraph{\bfseries Transport approximation for $r \geq R$} 

We now state an asymptotic result that holds for position parameters $r \geqslant R$.
Introducing the set 
\begin{equation*}
    \T_R^+ := \{ \icdf_{\ssol},~ r\geq R\} \subset L^2(0,1), 
\end{equation*}
we have the following result.
\begin{theorem} \label{thrm:transp_after}
There exists a constant $C > 0$ independent of $R$ such that for all $n\in\mathbb{N}^*$,
	\begin{equation*}
		{\rm d}_{\infty, 4}(\T_R^+, \L^2(0, 1)) \leqslant C R e^{-\frac{1}{2} \zeta_R R}.
	\end{equation*}
\end{theorem}

The proof of Theorem~\ref{thrm:transp_after} is postponed to Section~\ref{sec:proof_thrm_transp_after}.

\bigskip
\paragraph{\bfseries Transport approximation for $r \in \RR^+$}
Now consider the unbounded set of 
inverse cumulative distribution function of solutions
\begin{equation*}
    \T := \{ \icdf_{\ssol},~ r\in\mathbb{R^+} \}.
\end{equation*}
We finally state our final result on the Kolmogorov $n$-width decay of $\mathcal T$.

\begin{theorem}\label{thrm_final}
	For all $\varepsilon > 0$, there exists a constant $C_\varepsilon>0$ such that for all $n \in \mathbb{N}^*$,
	\begin{equation*}
		\dn(\T, \L^2(0, 1)) \leqslant C_\varepsilon n^{\frac{25}{\frac{21}{2} - \frac{\varepsilon}{z}} - \frac{5}{2}}.
	\end{equation*}
\end{theorem}

The proof of Theorem~\ref{thrm_final} is postponed to Section~\ref{sec:proof_thrm_final}.

\subsection{On the mixture Wasserstein transport metric: case \texorpdfstring{$\mathbb{M} = \mathcal {\rm SM}(\mathbb{R})$}{M = P2(R)} with \texorpdfstring{$\delta = {\rm MW}_2$}{d = W2}}

We now provide theoretical results related to this mixture distance and comment on the theoretical results.

\subsubsection{Preliminary discussion}

We first specify the definitions of metric Kolmogorov widths considered in the mixture Wasserstein framework of our theoretical analysis.

\medskip

First the computation of an exact mixture Wasserstein barycenter using formula \eqref{eq:bar_mixtures} requires, for any collection of barycentric weights $\boldsymbol{\lambda}\in \Lambda_n$ and of Slater mixtures $\boldsymbol{m}\in {\rm SM}(\mathbb{R})^n$, the computation of a solution $w^*({\bm m}; {\bm \lambda})$ to the minimization problem~\eqref{eq:mMW2}. 
In our model order reduction context, this appears to be too costly from a computational point of view. 
For this reason we rather rely 
on the use of {\itshape approximate } mixture barycenters, the computation and definition of which do not require the resolution of multiple problems of the form ~\eqref{eq:mMW2}. More precisely, denoting by $\overline{\boldsymbol \lambda}:=(1/n, 1/n, \ldots, 1/n)\in \Lambda_n$, for any (possibly nonunique) solution $w^*\left({\bm m}; \overline{{\bm \lambda}}\right)$ of~\eqref{eq:mMW2}, we define (recall formula~\eqref{eq:bar_mixtures} for the exact mixture Wasserstein barycenter)
\begin{equation}\label{eq:approxBary}
        \appbarMW\left({\bm m}\right) :=
        \sum_{\kbar \in \mathbf{K}} \left(w^*({\bm m}; \overline{{\bm \lambda}})\right)_{\kbar} \barW\left(m^1_{k^1}, \dots, m^n_{k^n}\right).
\end{equation}
Note that it is easy to check that the approximate mixture barycenter remains an interpolation between the mixtures $m^1, \ldots, m^n$ in the following sense: if ${\bm \lambda} = (\delta_{n_0, i})_{1\leq i \leq n}$ for some $1\leq n_0 \leq n$, then $ \appbarMW\left({\bm m}\right) = m^{n_0}$.

\medskip

In addition, we observed in our numerical tests that restricting the set of weights $\boldsymbol{\lambda}$ to the set $\Lambda_n$ of barycentric weights is quite restrictive. Additionally, the definition of approximate mixture barycenters can easily be extended to the case where the collection of weights $\boldsymbol{\lambda}$ belongs to the larger set 
    \begin{equation}
    \Omega_n(\boldsymbol{m}) = \left\{ \bm \lambda \in \RR^n, ~ \forall \kbar \in \{1, \dots, K^1\} \times \dots\times \{1, \dots, K^n\}, ~ \sum_{i = 1}^n \frac{\lambda_i}{\zeta^i_{k^i}} > 0 \right\}.
    \end{equation}
Indeed, this requires to extend the definition of Wasserstein barycenters between a collection of Slater distributions, which we do via \eqref{eq:barySlater} using \eqref{eq:formula}, since the quantities defined in \eqref{eq:formula} are still well-defined for $\boldsymbol{\lambda}\in\Omega_n(\boldsymbol{m})$. In the following, for all $\boldsymbol{\lambda}\in \Omega_n(\boldsymbol{m})$, we denote by $\overline{\mathcal B}_{\distMW}^{\bm \lambda}({\bm m})$ the set of approximate mixture Wasserstein barycenters $\appbarMW\left({\bm m}\right)$ which can be written under the form~\eqref{eq:approxBary}.

\bigskip

We are now in position to consider two following quantities, which are related to (but different from) the metric Kolmogorov widths in the mixture Wasserstein space. We thus define the {\itshape extended approximate metric $\L^\infty$ Kolmogorov width}
    \begin{equation}
        \label{eq:kolmo_mw2}
        \dtildeinfn (\mathcal{E}, SM(\RR)) = \inf_{{\bm m}\in SM(\RR)^n}
            \sup_{z\in\mathbf{Z}}~\inf_{\bm\lambda\in \Omega_n({\bm m})} ~\inf_{b \in \overline{\mathcal B}_{\distMW}^{\bm \lambda}({\bm m})} \distMW (u(z), b).
    \end{equation}
and the {\itshape extended approximate metric $\L^2$ Kolmogorov width}
    \begin{equation*}
        \dtildedn(\mathcal{E}, SM(\RR)) := \inf_{{\bm m}\in SM(\RR)^n}
            \left(\int_{z\in\mathbf{Z}}~ ~\inf_{\bm\lambda\in \Omega_n({\bm m})} ~\inf_{b \in \overline{\mathcal B}_{\distMW}^{\bm \lambda}({\bm m})} \distMW (u(z), b)^2 \right)^{1/2}.
    \end{equation*}

\subsubsection{Main result}

In this part, we also focus on the $M = 2$ symmetric case, i.e. where $\pos = (-r, r)$ and with fixed equal
charges $\charges = (z, z)$. Once again, we denote by $\M_\charges$ the set of solutions.
In that case, the extended approximate metric Kolmogorov $n$-width relative to the mixture distance $\distMW$~\eqref{eq:mixture_distance} defined in~\eqref{eq:kolmo_mw2} is zero for $n \ge 2$.
This means that in that case, any solution can be exactly be obtained from a collection only two Slater mixtures.

\begin{theorem}\label{thrm_mixture}
    Let $\mathrm{SM}(\RR) \subset \P_2(\RR)$ be the metric space of mixtures of Slater distributions defined in~\eqref{eq:SMR} endowed
    with its metric $\distMW$ defined in~\eqref{eq:mixture_distance}.
    For all $n > 1$, there holds
    \begin{equation*}
         \dtildeinfn(\M_\charges, \mathrm{SM}(\RR))  = \; \dtildedn(\M_\charges, \mathrm{SM}(\RR)) = 0. 
    \end{equation*}
\end{theorem}

The proof of Theorem~\ref{thrm_mixture} is postponed to Section~\ref{sec:proof_thrm_mixture}.

\subsection{Link with other types of Kolmogorov widths}

We now compare the proposed extended metric Kolmogorov widths with two other notions of generalized Kolmogorov widths.

\subsubsection{Link with nonlinear Kolmogorov widths}

Let us point out that the metric Kolmogorov width $    \dinfn(\mathcal{E}, \mathcal P_2(\mathbb{R}))$ can be related to nonlinear Kolmogorov widths~\cite{temlyakov1998nonlinear,devore1998nonlinear}. Indeed this quantity can be rewritten as
\begin{equation}\label{eq:nonlinear}
\mathop{\inf}_{\boldsymbol{m} \in \mathbb{M}^n} \sup_{z\in\mathbf{Z}} ~ \delta\left( u(z), D_{\boldsymbol{m}} \circ C_{\boldsymbol{m}}(u(z)) \right),
\end{equation}
with $\mathbb{M} = \mathcal P_2(\mathbb{R})$, $\delta  = W_2$, and where for all $\boldsymbol{m}\in \mathbb{M}^n$, $D_{\boldsymbol{m}}$ refers to a particular decoder map and $C_{\boldsymbol{m}}$ a particular encoder map defined as
$$
D_{\boldsymbol{m}}: \Lambda_n \ni \boldsymbol{\lambda} \mapsto \bary_{W^2}({\bm m})\in \mathcal P_2(\mathbb{R})
\quad \mbox{ and } \quad
C_{\boldsymbol{m}}: \mathcal P_2(\mathbb{R}) \ni u \mapsto \mathop{\rm argmin}_{\boldsymbol{\lambda} \in \Lambda_n} W_2(u, \bary_{W^2}({\bm m}))\in \Lambda_n.
$$

However, in the case of  $   \dtildeinfn (\mathcal{E}, {\rm SM}(\RR)) $, the situation is more subtle since there is in general no uniqueness of approximate mixture Wasserstein barycenters. It is therefore not obvious to rewrite $\dtildeinfn (\mathcal{E}, {\rm SM}(\RR))$ under the form (\ref{eq:nonlinear}) with $\mathbb{M} = {\rm SM}(\RR)$, $\delta  = MW_2$ for some decoder maps $D_{\boldsymbol{m}}$ and encoder maps $C_{\boldsymbol{m}}$ parametrized by $\boldsymbol{m}\in \mathbb{M}^n$.

\subsubsection{Link with Kolmogorov $(n, m)$-
width}

The extended metric Kolmogorov width can also be linked to 
the Kolmogorov $(n,m)$-width introduced in~\cite{rim2023manifold} as follows.
Denoting by $\mathbb{T}$ the set of bijective maps from $\mathbb{R}$ to $\mathbb{R}$, the Kolmogorov $(n,m)$-width then reads as:
$$
d_{n,m}(\mathcal{E}, \mathbb{H})
            := \inf_{\substack{V_n\subset \mathbb{H}\\\dim V_n=n\\
            W_m \subset \mathbb{T}\\
            \dim\; {\rm Span}\{W_m\}}=m\\
            }\sup_{z\in \mathbf{Z}}~ \inf_{\substack{v_n\in V_n \\ T_m \in {\rm Span}\{W_m\} \cap \mathbb{T}}} \Vert u(z)-v_n\circ T_m^{-1}\Vert.
$$

In the present setting, it is more convenient to measure errors with respect to the Wasserstein metric. This is our motivation for introducing the quantity
$$
d_{1,m,W_2}(\mathcal{E}, \mathcal P_2(\mathbb{R}))
            := \inf_{\substack{v \in \mathcal P_2(\mathbb{R})\\
            W_m\subset \mathbb{T}\\
            \dim \; {\rm Span}\{W_m\}=m\\
            }}\sup_{z\in \mathbf{Z}}~ \inf_{T_m \in \rm Span }\{W_m\} \cap \mathbb{T} W_2(u(z), v\circ T_m^{-1}).
$$
In the case when $\boldsymbol{m} = (m_1, \ldots, m_n)\in \mathcal P_2(\mathbb{R})$ is a family of $n$ probability measures which are absolutely continuous with respect to the Lebesgue measure and the support of which is equal to the whole space $\mathbb{R}$, then it holds that, for all $\boldsymbol{\lambda} = (\lambda_1, \ldots, \lambda_n)\in \Lambda_n$,  
$$
\bary_{W_2}({\boldsymbol m}) =  \left( \sum_{i=1}^n \lambda_i T_i \right) \# m_1,
$$
where for all $1\leq i \leq n$, $T_i$ denotes the optimal transport map between $m_1$ and $m_i$. Thus, in the case of the exact Wasserstein metric, if the probability measures are regular enough and have full support in~$\mathbb{R}$, $\dinfn(\mathcal{E}, \mathcal P_2(\mathbb{R}))$ is equal to $d_{1,n,W_2}(\mathcal{E}, \mathcal P_2(\mathbb{R}))$. However, when the probability measures are not regular enough, or in the case of the mixture Wasserstein metric, the existence of optimal transport maps is no longer guaranteed, which makes difficult to draw a link between the Kolmogorov
width considered in the present work, and the $(n,m)$-widths of~\cite{rim2023manifold}.

\section{Nonlinear reduced basis method}
\label{sec:nonlinearRB}

Motivated by the fast decay of the generalized Kolmogorov $n$-width for the Wasserstein mixture distance, we now propose a nonlinear reduced basis method for this problem. The method, as is common in reduced order modeling, is based on an 
offline phase followed by an online phase. In the
offline phase a few representative snapshots are selected thanks to a greedy algorithm. In the online phase, for any new set of parameters, the energy of the system is minimized over the set of barycenters of selected snapshots, using a quasi-Newton minimization algorithm started at several initial points. Let us emphasize here that, despite the fact that the theoretical results stated in Section~\ref{sec:kolmogorov} only hold for the one-dimensional toy problem introduced in Section~\ref{sec:modelproblem}, the numerical strategy presented in this section can be used to build nonlinear reduced order models for more general problems.

\subsection{Greedy algorithm} \label{sec:greedy}

We first present the greedy algorithm used in the offline phase.
Let $\charges \in \left(\RR_+^*\right)^M$ be fixed positive charges, $\M^I_\charges = \{ \sol, ~ \pos \in I^M \}$ with $I \subset \RR$ an interval
be a set of solutions of~\eqref{prob:schr_main} and $\M_{tr} \subset \M^I_\charges$ a finite  training set of already computed solutions
called snapshots.
The aim here is to select the most representative snapshots in $ \M_{tr}$, so that any solution $u \in \M^I_\charges$ can be efficiently approximated with only a few snapshots.
The main idea in this greedy algorithm is to select at each iteration the snapshot in $\M_{tr}$ the approximation of which as a mixture barycenter of previously selected snapshots leads to the highest error.
Since the proposed algorithm is generic to any training set where the elements can be represented by mixtures equipped with a mixture distance, we use the notation $m$ for the elements in $\M_{tr}$ instead of~$u$.
Recall from Section~\ref{sec:OT} that we write a mixture as $m = \displaystyle \sum_{k = 1}^K \pi_k m_k$, so for the rest of this section, we denote an element of $\M_{tr}$ as $m$.
In our case, these mixtures are solutions to problem~\eqref{prob:schr_main}, which means that their elements, denoted by $m_k$ are Slater functions with  a scale parameter $\zeta$ independent of $k$ and position parameter $r_k$.
In the following, mixtures written as $m^i$ where $i$ is an integer follow the same rules of notation, their parameters being $\left(\pi_{k^i}^i\right)_{k^i=1}^{K^i}$ for the weights, $\zeta^i$ for the common scale parameters and $\left(r_{k^i}^i\right)_{k^i=1}^{K^i}$ for the position parameters.
The proposed greedy algorithm is presented below (Algorithm~\ref{algorithm:greedy}).

\begin{algorithm}
    \caption{\textsc{Greedy algorithm}}
    \label{algorithm:greedy}
    \begin{algorithmic}
        \STATE {\bf Input:} $\M_{tr}$, training set; $N$, number of elements to select  
            \STATE{Select $m^1$ and $m^2$ solutions to $\displaystyle \mathop{\argmax}_{(m^1, m^2) \in \M_{tr}} \distMW(m^1, m^2)^2$.}
            \STATE{$\mathcal{B} := \{m^1, m^2 \}$}
        \FOR{$n = 2, \dots, N-1$}
            \STATE{
                Select
                \begin{equation}
                    \label{eq:greedy_eq}
                    m^{n+1} \in \argmax_{m \in \M_{tr}} \inf_{\bm \lambda \in \Omega_n(m^1,\ldots, m^n)} \distMW\left(m, \appbarMW(m^1, \dots, m^n)\right)^2,
                \end{equation}
                where
                    \begin{equation*}
                        \Omega_n(m^1,\ldots, m^n) = \left\{ \bm\lambda \in \RR^n, ~
                        \sum_{i=1}^n \frac{\lambda_i}{\zeta^i} > 0 \right\}.
                    \end{equation*}
            } 
            \STATE{$\mathcal{B} = \mathcal{B} \cup \{ m^{n+1} \}$}
        \ENDFOR
        \STATE {\bf Output:} Reduced basis $\mathcal{B} \subset \M_{tr}$ 
    \end{algorithmic}
\end{algorithm}

The keystone of Algorithm~\ref{algorithm:greedy} is the resolution of problem~\eqref{eq:greedy_eq}, and more precisely for $m^1,\ldots, m^n\in\M_{tr}$ and $m\in \M_{tr}$, the resolution of the following minimization problem
\begin{equation} \label{prob:min_greedy}
    \inf_{\bm \lambda \in \Omega_n(m^1,\ldots, m^n)}\distMW\left(m, \appbarMW(m^1, \dots, m^n)\right)^2.
\end{equation}
In practice, we start by solving problem~\eqref{eq:mMW2} with ${\bm \lambda} = \overline{{\bm \lambda}} = (1/n, \ldots, 1/n)$ to obtain the barycenters weights $w^*({\bm m}, \overline{{\bm \lambda}})$ appearing in~\eqref{eq:approxBary} for the calculation of $\appbarMW(m^1, \dots, m^n)$.
It can be computed using any linear programming solver as problem~\eqref{eq:mMW2} is a linear problem with linear constraints. In the sequel, for the sake of simplicity, we denote by $w^* := w^*({\bm m}, \overline{{\bm \lambda}})$.
Note that in the representation of the barycenters~\eqref{eq:bar_mixtures}, one can in fact only consider the indices $\kbar = (k^1,k^2,\ldots,k^n) \in \mathbf{K} \subset \{1, \dots, K^1\} \times \dots \times \{1, \dots, K^n\}$ for which $\wstar_{\kbar}$ is non zero (see Remark~\ref{rem:sparsity})  to reduce the dimensionality of the problem.

Now, by the definition of the distance $\distMW$, we can say that
\begin{equation*}
    \distMW\left(m, \appbarMW(m^1, \dots, m^n)\right)^2
    = \min_{w \in \Pi(\bm \pi, w^*)} \sum_{\kbar \in \mathbf{K}} \sum_{k = 1}^K
    w_{\kbar, k} \distW\left(m_k, \barW(m^1_{k^1}, \dots, m^n_{k^n})\right)^2,
\end{equation*}
where the weights $w$ are matrices with non-negative terms indexed by $\kbar \in \mathbf{K}$ and $k \in \{1, \dots, K\}$ in the set
\begin{equation}
    \label{eq:forthevertices}
   \Pi(\bm \pi, w^*) := 
   \left\{ 
   w \in (\RR_+)^{|\mathbf{K}| \times K}, \;
   \forall k \in \{1, \dots, K\}, ~ \sum_{\kbar \in \mathbf{K}} w_{\kbar, k} = \pi_k,
    \forall \kbar \in \mathbf{K}, ~ \sum_{k = 1}^{K} w_{\kbar, k} = w^*_{\kbar}
    \right\}.
\end{equation}
Note that the minimization set $\Pi(\bm \pi, \wstar)$ does not depend on the parameter $\bm \lambda$ as the previously computed weights of barycenter $\wstar$ do not depend on $\bm \lambda$ either. We also have, recalling that the parameters for the mixture $m$ are $r_k$ and $\zeta$
\begin{equation*}
    \distW\left(m_k, \barW(m^1_{k^1}, \dots, m^n_{k^n})\right)^2 = \left( r_k - \sum_{i = 1}^n\lambda_i r^i_{k^i} \right)^2
        + 2\left( \frac{1}{\zeta} - \sum_{i = 1}^n \frac{\lambda_i}{\zeta^i} \right)^2
        = {\bm\lambda}^\intercal A_{\kbar} {\bm\lambda} + b_{\kbar, k}^\intercal {\bm\lambda} + c_k,
\end{equation*}
where
\begin{equation*}
    A_{\kbar} = \mathbf{r}_{\kbar}^\intercal \mathbf{r}_{\kbar} + 2\bm\zeta^\intercal \bm\zeta,
    \qquad
    b_{\kbar, k} = -2\left( r_k \mathbf{r}_{\kbar} + \frac{2}{\zeta} \bm\zeta \right),
    \qquad \text{and} \qquad
    c_k = r_k^2 + \frac{2}{\zeta^2},
\end{equation*}
with $\mathbf{r}_{\kbar} = \left( r^1_{k^1}, \dots, r^n_{k^n} \right)$ and $\bm\zeta = \left( \frac{1}{\zeta^1}, \dots, \frac{1}{\zeta^n} \right)$.
In particular the matrices $A_{\kbar}$ are non-negative for any $\kbar \in \mathbf{K}$.
Hence, problem~\eqref{prob:min_greedy} reduces to
\begin{equation} \label{prob:min_greedy:double_quad_form}
    \inf_{\bm \lambda \in \Omega_n(m^1,\ldots, m^n)} \min_{w \in \Pi(\bm\pi, w^*)}
    {\bm\lambda}^\intercal A {\bm\lambda} + b_w^\intercal {\bm\lambda} + c
    =  \min_{w \in \Pi(\bm\pi, w^*)} \inf_{\bm \lambda \in \Omega_n(m^1,\ldots, m^n)}
    {\bm\lambda}^\intercal A {\bm\lambda} + b_w^\intercal {\bm\lambda} + c,
\end{equation}
where 
\begin{equation} \label{eq:coeff_proj_mini}
    A = \sum_{\kbar \in \mathbf{K}} w^*_{\kbar} A_{\kbar},
    \qquad
    b_w = \sum_{\kbar \in \mathbf{K}} \sum_{k = 1}^K w_{\kbar, k} b_{\kbar, k},
    \qquad\text{and}\qquad
    c = \sum_{k = 1}^K \pi_k c_k.
\end{equation}
The matrix $A$ is also non-negative as a sum of non-negative matrices.
Note that since the matrix $A$ does not depend neither on $\bm \lambda$ and on the mixture $m$, we can indeed compute $A$ at each update of $\mathcal{B}$, just like the weights $\wstar$. In all our numerical tests, we have numerically checked that the matrix $A$ is in fact positive definite.
Then, the solution to the minimization problem
\begin{equation*}
    \min_{\bm \lambda \in \RR^n} {\bm\lambda}^\intercal A {\bm\lambda} + b_w^\intercal {\bm\lambda} + c
\end{equation*}
is  ${\bm \lambda}_w = -\frac{1}{2}A^{-1} b_w$,
and we can also check here a posteriori that the solution ${\bm \lambda}_w\in \Omega_n(\boldsymbol{m})$, so that the solution of $\inf_{\bm \lambda \in \Omega_n(\boldsymbol{m})}
    {\bm\lambda}^\intercal A {\bm\lambda} + b_w^\intercal {\bm\lambda} + c$ is also ${\bm \lambda}_w$, which is always the case in the tested examples. If however it turned out not to be the case, since $\Omega_n(\boldsymbol{m})$ is a convex set, it is possible to directly solve the minimization problem on $\Omega_n$ using quadratic programming. In the case when $A$ might not be invertible (which never occured in the numerical tests performed in this work), the inverse of $A$ would naturally be replaced by its pseudo-inverse in the expression of ${\bm \lambda}_w$ above, and in the expressions below.

Hence, by putting ${\bm \lambda}_w$ back in problem~\eqref{prob:min_greedy:double_quad_form},
we have that problem~\eqref{prob:min_greedy} is equivalent to
\begin{equation*}
    \min_{w \in \Pi(\bm\pi, \wstar)} - \frac{1}{4} b_w^\intercal A^{-1} b_w + c \; = \min_{w \in \Pi(\bm\pi, \wstar)} w^\intercal \left(- \frac{1}{4} B^\intercal A^{-1} B \right) w + c,
\end{equation*}
by choosing a vectorization for the weights $w$ and where $B$ is such that $Bw = b_w$,
which is a concave quadratic minimization problem because the matrix $- \frac{1}{4} B^\intercal A^{-1} B$ is negative since $A$ is positive.
The matrix $- \frac{1}{4} B^\intercal A^{-1} B$ of this problem has a size $K |\mathbf{K}|$. For a hint on the size $|\mathbf{K}|$, see Remark~\ref{rem:sparsity}.
The solution of the problem is in fact a vertex of the polytope $\Pi(\bm\pi, \wstar)$, thanks to the convexity of the polytope and the concavity of the problem~\cite{Floudas1995-ue}.
We summarize the whole procedure to solve problem~\eqref{eq:greedy_eq} in Algorithm~\ref{algorithm:offline_minimization} below.

\begin{algorithm}
    \caption{\textsc{Offline projection minimization}}
    \label{algorithm:offline_minimization}
    \begin{algorithmic}
        \STATE {\bf Input:} $\mathcal{B} = \{m^1, \dots, m^n\}$, selected elements
        \STATE{Compute $\wstar$ as in problem~\eqref{eq:mMW2}.}
        \STATE{Compute the matrix $A$ as in~\eqref{eq:coeff_proj_mini}, check that it is positive definite, and compute $A^{-1}$.}
        \STATE{
            Select
            \begin{equation*}
                m^{n+1} \in \argmin_{m \in \M_{tr}} \min_{w \in \Pi(\bm\pi, \wstar)} - \frac{1}{4} b_w^\intercal A^{-1} b_w + c,
            \end{equation*}
            where $b_w$ and $c$ are given in~\eqref{eq:coeff_proj_mini}.
        } 
        \STATE {\bf Output:} $m^{n+1}$
    \end{algorithmic}
\end{algorithm}

\begin{remark}
    In our implementation, we took advantage of the concave setting of the problem and searched the solution directly among the vertices of the polytope $\Pi(\bm\pi, \wstar)$ to ensure global optimality.
    In practice, global optimization packages such as Gurobi could be used.
\end{remark}

\begin{remark}
    Algorithm~\ref{algorithm:greedy} is in fact more general and can be used wherever the elements of the training set can be represented by mixtures. For example, we can consider a setting where the solutions are Slater mixtures with different scale parameters $\zeta$.
    In this scenario, the set of admissible weights for the barycenters become a bit more complex and reads
    \begin{equation*}
        \Omega_n(\boldsymbol{m}) = \left\{ \bm\lambda \in \RR^n, ~\forall \kbar \in \{1, \dots, K^1\} \times \dots \times \{1, \dots, K^n\}, ~ \sum_{i=1}^n \frac{\lambda_i}{\zeta^i_{k^i}} > 0 \right\}.
    \end{equation*}
    More generally, one can consider any mixtures for which a Wasserstein mixture distance is well-defined, as presented in~\cite{Dusson2023-ah}, which includes e.g. Gaussian mixtures, 
     upon modifying the set $\Omega_n(\boldsymbol{m})$ with the correct parameters of the distributions to ensure admissibility of barycenters.
\end{remark}

\subsection{Online algorithm}

Once the reduced basis is computed, we want to efficiently compute approximations of solutions, given a new position for the nuclei. 
Since the projection minimization algorithm used in the offline phase requires  a high-fidelity estimate of the exact solution, it is not a viable option for the online phase.
 Here we instead take advantage of the structure of our problem, which is an energy minimization problem~\eqref{eq:energy_min}, and we minimize the energy of the new system over the set of barycenters of the elements in the reduced basis.
More precisely, assume that we selected $N$ mixtures solutions $m^1, \dots, m^N$ in the offline phase and we want to obtain an approximation to the solution with molecular parameters $\pos$.
We consider the following optimization problem 
\begin{equation}
    \label{eq:minEonline}
     \inf_{\bm \lambda \in \Omega_N(m^1,\ldots,m^N)} \frac{\Esol\left( \appbarMW(m^1, \dots, m^N) \right)}{\left\|\appbarMW(m^1, \dots, m^N)\right\|^2_{L^2(\mathbb{R})} }, 
\end{equation} 
where $\Omega_N(m^1,\ldots,m^N)$ is the extended set of admissible barycenters, as in the greedy algorithm.
For clarity, we write from now on $\Esol(\bm \lambda)$ instead of $ \frac{\Esol\left( \appbarMW(m^1, \dots, m^N) \right)}{\left\|\appbarMW(m^1, \dots, m^N)\right\|^2_{L^2(\mathbb{R})}}$.
Note that the energy functional can be easily computed with the following formula 
\begin{align}\label{eq:Elambda}
        & \Esol(\bm \lambda) = \left[ {\sum_{\kbar \in \bm K}\sum_{\lbar \in \bm K} \wstar_{\kbar} \wstar_{\lbar} (1 + \zetabar|\rbar - \rbar[\lbar]|) e^{- \zetabar|\rbar - \rbar[\lbar]|}}  \right]^{-1} . \\
        & \left( \frac{\zetabar^2}{2} \sum_{\kbar \in \bm K}\sum_{\lbar \in \bm K} \wstar_{\kbar} \wstar_{\lbar} (1 - \zetabar|\rbar - \rbar[\lbar]|) e^{- \zetabar|\rbar - \rbar[\lbar]|} 
        -  \zetabar \sum_{m = 1}^M z_m \sum_{\kbar \in \bm K} \wstar_{\kbar} e^{- \zetabar|\rbar - r_m|} \right) \nonumber
   \end{align}
where $\zetabar = \left[ \sum_{i = 1}^N \frac{\lambda_i}{\zeta^i} \right]^{-1}$ is the scale parameter and $\rbar = \sum_{i = 1}^N \lambda_i r^i_{k^i}$ is the position parameter of a Slater component of a barycenter with weights $\bm \lambda$, and $w^*$ is the solution to problem~\eqref{eq:mMW2} for the reduced basis, and can be computed offline.

\begin{remark}
Let us point out that the minimization problem~(\ref{eq:minEonline}) is in general non-convex. Numerical solutions obtained by the minimization algorithms used here then depends on the initialization. 
\end{remark}

Since the energy functional as a function of $\bm \lambda$ is nonconvex, and in practice exhibits many local minima, 
solving problem~\eqref{eq:minEonline} requires to use a global optimization algorithm, preferably very robust to ensure that the global minimizer is found, to guarantee repeatability of the results.
The natural optimization procedure is detailed here, and detailed in Algorithm~\ref{algorithm:online}. We use a quasi-Newton minimization algorithm (LBFGS) with evenly distributed starting points using a Sobol sequence on a representative set $B_N = [-B, B]^N \cap \Omega_N$ of values of $\bm \lambda$. 

\begin{algorithm}
    \caption{\textsc{Online optimization}}
    \label{algorithm:online}
    \begin{algorithmic}
        \STATE {\bf Input: }Reduced basis $m^1, \dots, m^N$, $B_N = [-B, B]^N \cap \Omega_N(m^1,\ldots, m^N)$, starting points $\bm \lambda_1, \dots, \bm \lambda_L$ in $B_N$
        \FOR{$l = 1, \ldots, L$ }
            \STATE{Compute  $\bm \lambda^*_l, E^*_l$ minimizer and energy solution found by optimizing $\Esol(\bm \lambda)$ for $\bm \lambda\in\Omega_N(m^1,\ldots, m^N)$ with starting point $\bm \lambda_l$ with a LBFGS algorithm} 
        \ENDFOR
        \STATE {\bf Output:} Minimizer $\bm \lambda^* = \argmin_{l=1,\ldots,L} E^*_l$
    \end{algorithmic}
\end{algorithm}

Note that the LBFGS algorithm requires the explicit computation of the gradient of the energy, which can easily be computed from formula~\eqref{eq:Elambda}. Also, to ensure that a solution in the constraint set $\Omega_N(m^1,\ldots, m^N)$ is found, points outside of $\Omega_N(m^1,\ldots, m^N)$ are penalized by
the function 
$\bm \lambda \longmapsto C + [\zetabar]^{-10}$,
where $C$ is a large positive constant to ensure a return in the domain $\Omega_N\Omega_N(m^1,\ldots, m^N)$ if a point outside of the constraint set is reached. Note that as the energy function explodes to $+\infty$ at the border of the domain, this is unlikely to happen.

\begin{remark} \label{remark:smoothing}
    To avoid failures in line searches in the LBFGS algorithm caused by the low regularity of the function $\Esol$, we actually consider a smoothed version of $\Esol$.
    To do so, we replace the absolute value (and its derivative the sign function) responsible for the low regularity in a small interval $[-\varepsilon, \varepsilon]$ with $\varepsilon > 0$, by a cosine function $x \longmapsto - \frac{2\varepsilon}{\pi} \cos\left(\frac{\pi}{2\varepsilon}x\right) + \varepsilon$. This new function is of class $\mathcal C^2$ on $\RR$. 
\end{remark}

\section{Numerical results}
\label{sec:num}

\begin{figure} \label{fig:example_solutions}
    \centering
    \includegraphics[width=.8\textwidth]{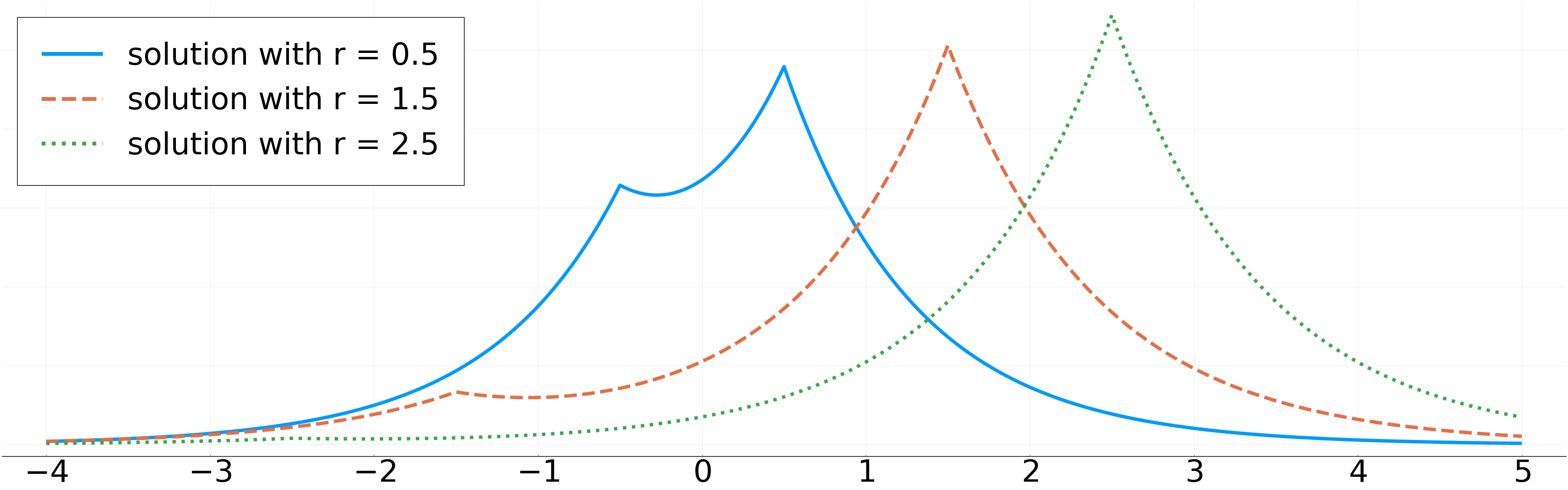}
    \vspace*{8pt}
    \caption{Three example solutions in $\M_{tr}$.}
\end{figure}

In this section, we present the numerical results obtained with the offline and online algorithms presented above. 
The code for generating the figures in the following can be found at \url{https://github.com/dussong/NonLinearReducedBasisOT}.
We focus on a system with two nuclei, i.e. $M=2$, and with charges $\charges = (0.8,1.1)$ and $\pos = (-r, r)$ for $r\in \RR_+$. For the training set, the $r$'s
are equally distributed on the interval $[0.5, 3]$ with $\sharp\M_{tr} = 251$. On Figure~\ref{fig:example_solutions}, we plot the exact solutions called snapshots for three examples, namely $r = 0.5, 1.5, 2.5$.

\subsection{Offline phase}
\label{sec:num:offline}

\begin{figure}
    \centering
    \includegraphics[width=.8\textwidth]{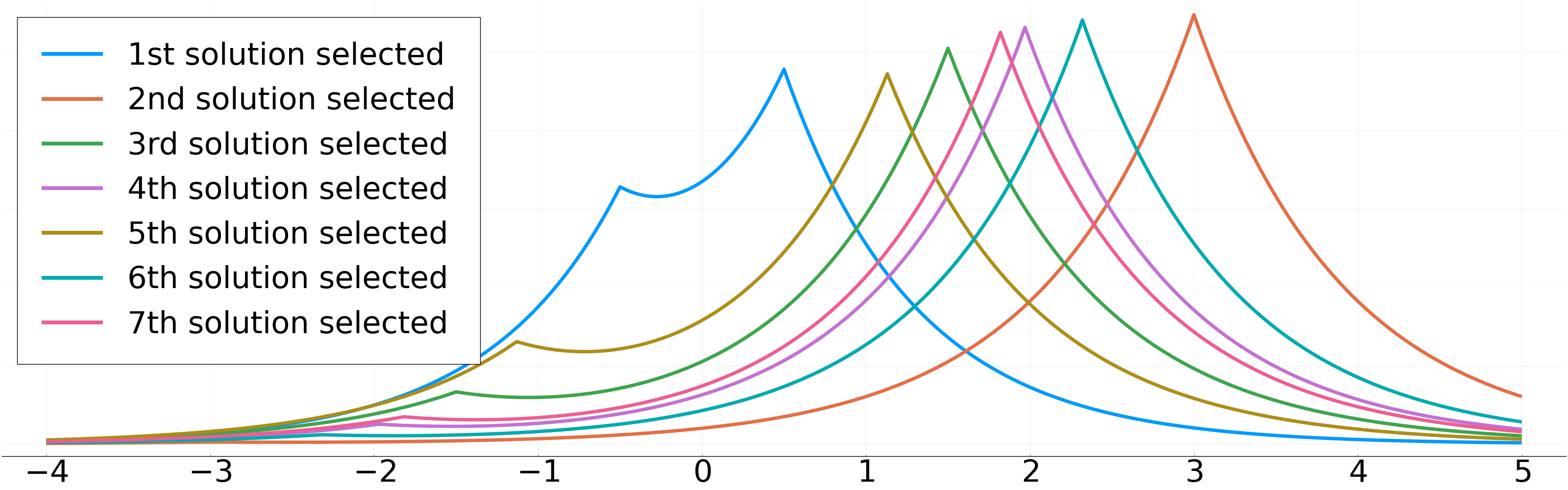}
    \vspace*{8pt}
    \caption{First seven elements selected in the reduced basis in the offline phase.}
    \label{fig:basis7}
\end{figure}

\begin{figure}
    \centering
    \includegraphics[width=0.6\textwidth]{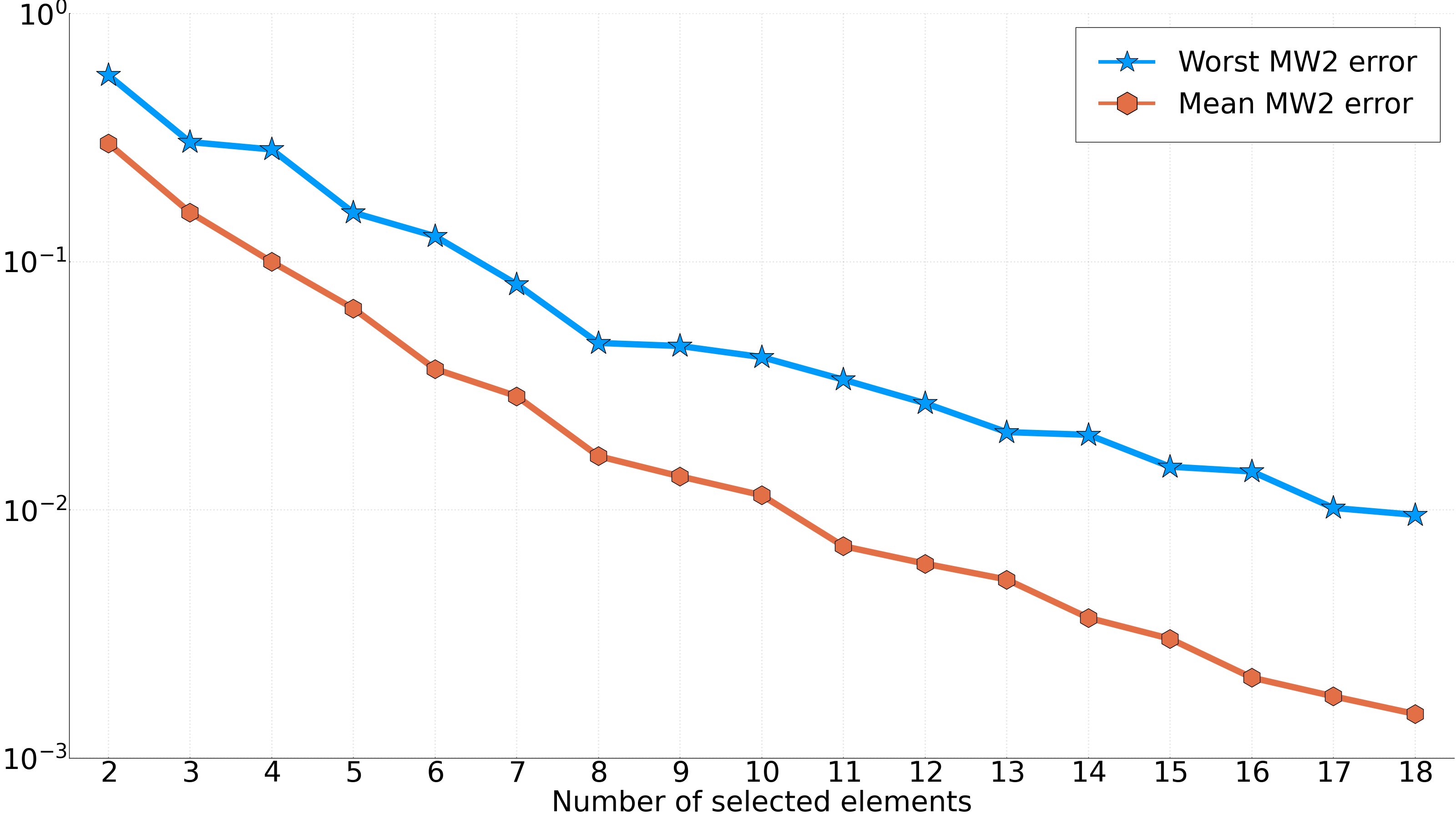}
    \vspace*{8pt}
    \caption{Decay of the projection error in the offline phase.}
    \label{fig:proj_decay}
\end{figure}

\begin{figure}[!htb]
    \centering
    \includegraphics[width=.6\textwidth]{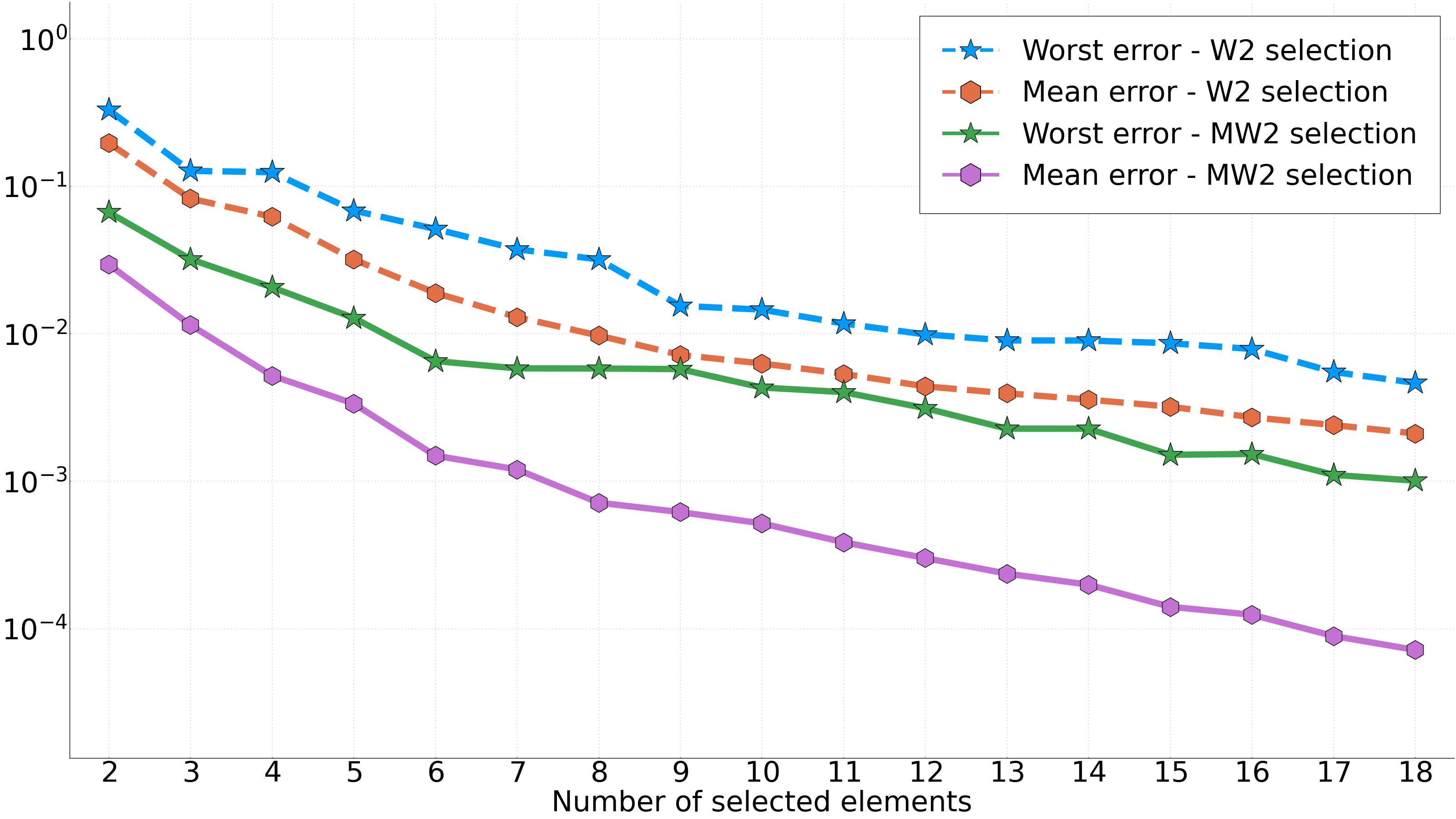}
    \vspace*{8pt}
    \caption{Decay of the projection error in $\distW$-distance in the offline phase for $\distMW$-distance and $\distW$-distance greedy selection.}
    \label{fig:proj_decay_w2}
\end{figure}

\begin{figure}
    \centering
    \includegraphics[width=0.35\textwidth]{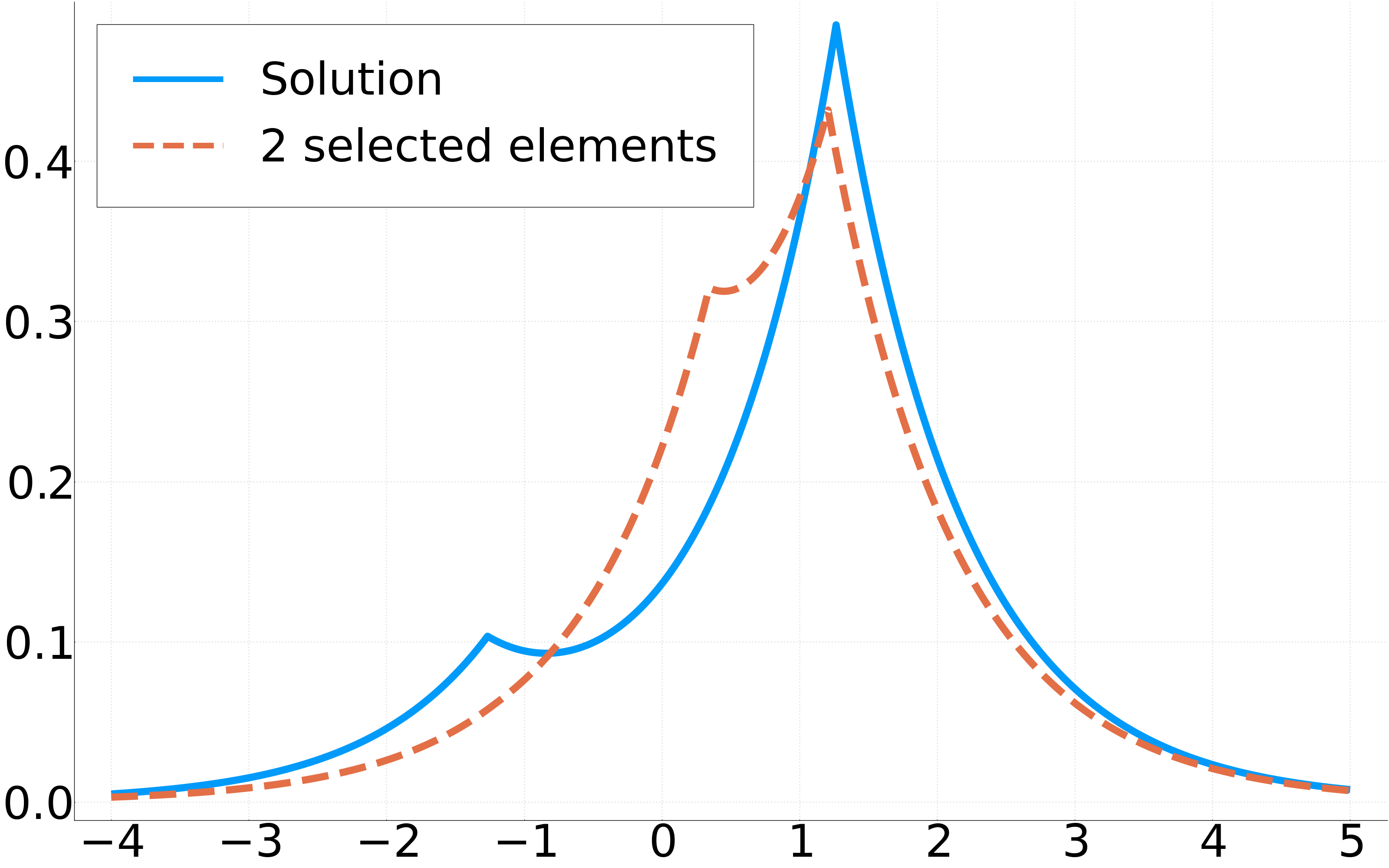}
    \includegraphics[width=0.35\textwidth]{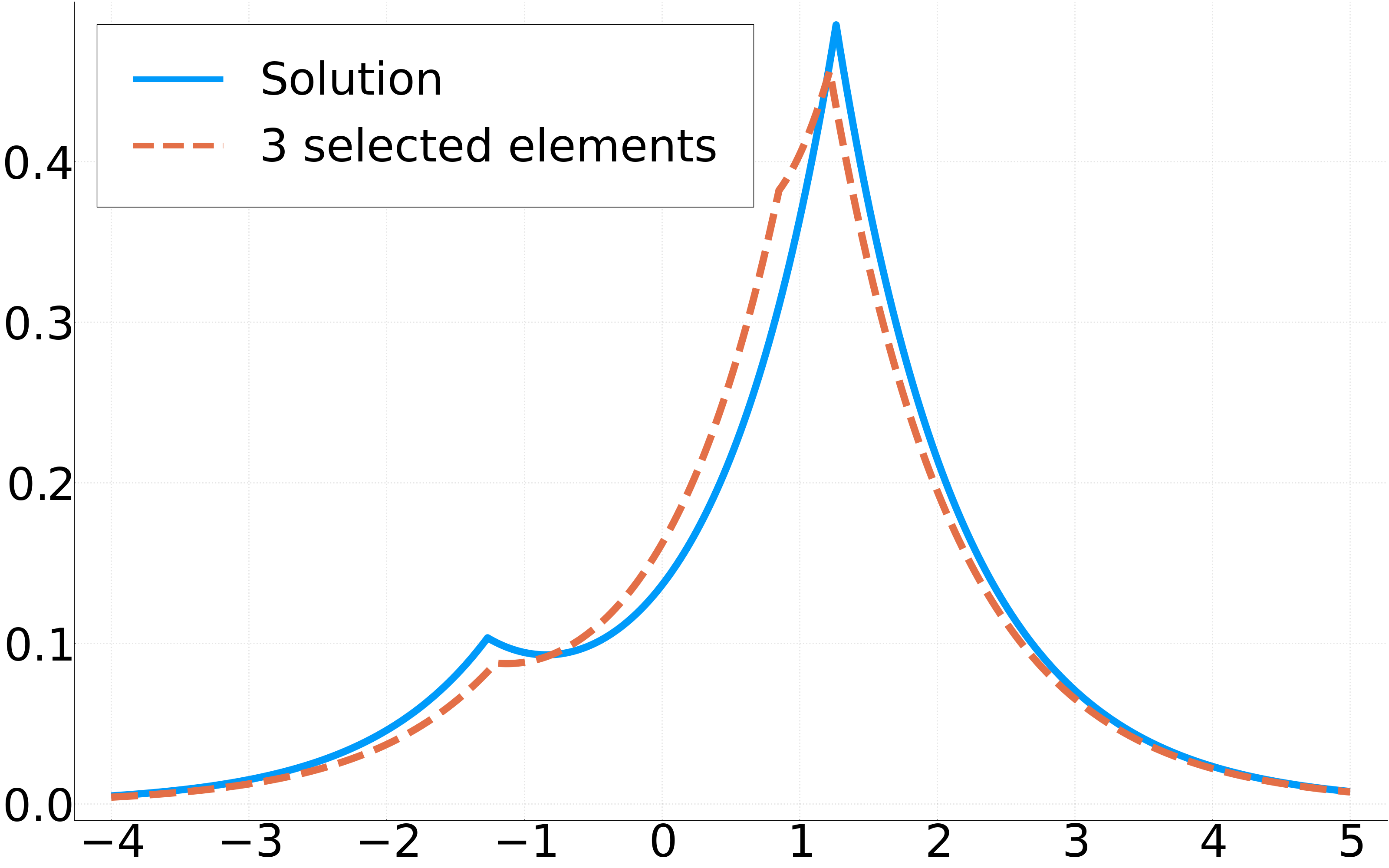}
    \includegraphics[width=0.35\textwidth]{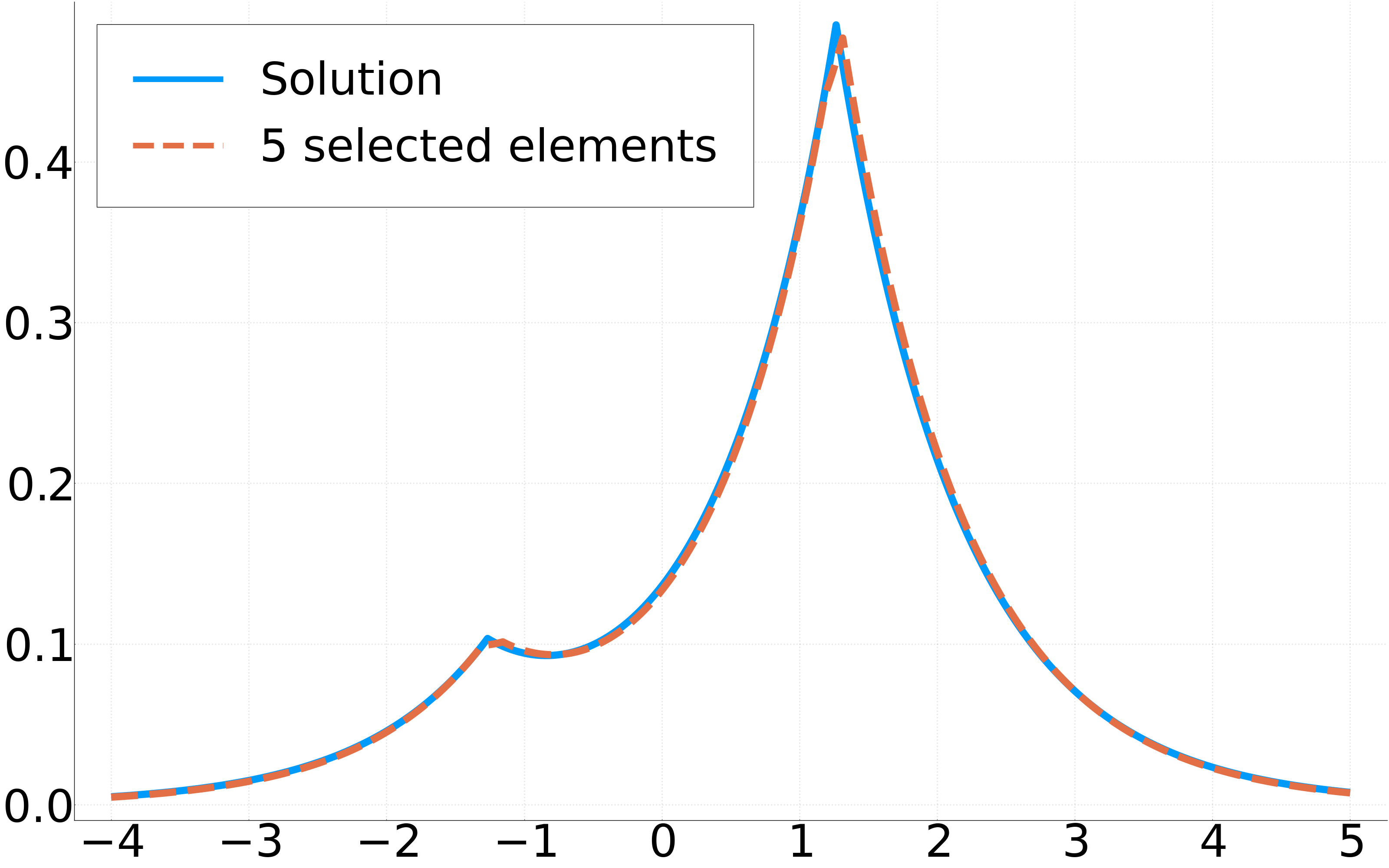}
    \includegraphics[width=0.35\textwidth]{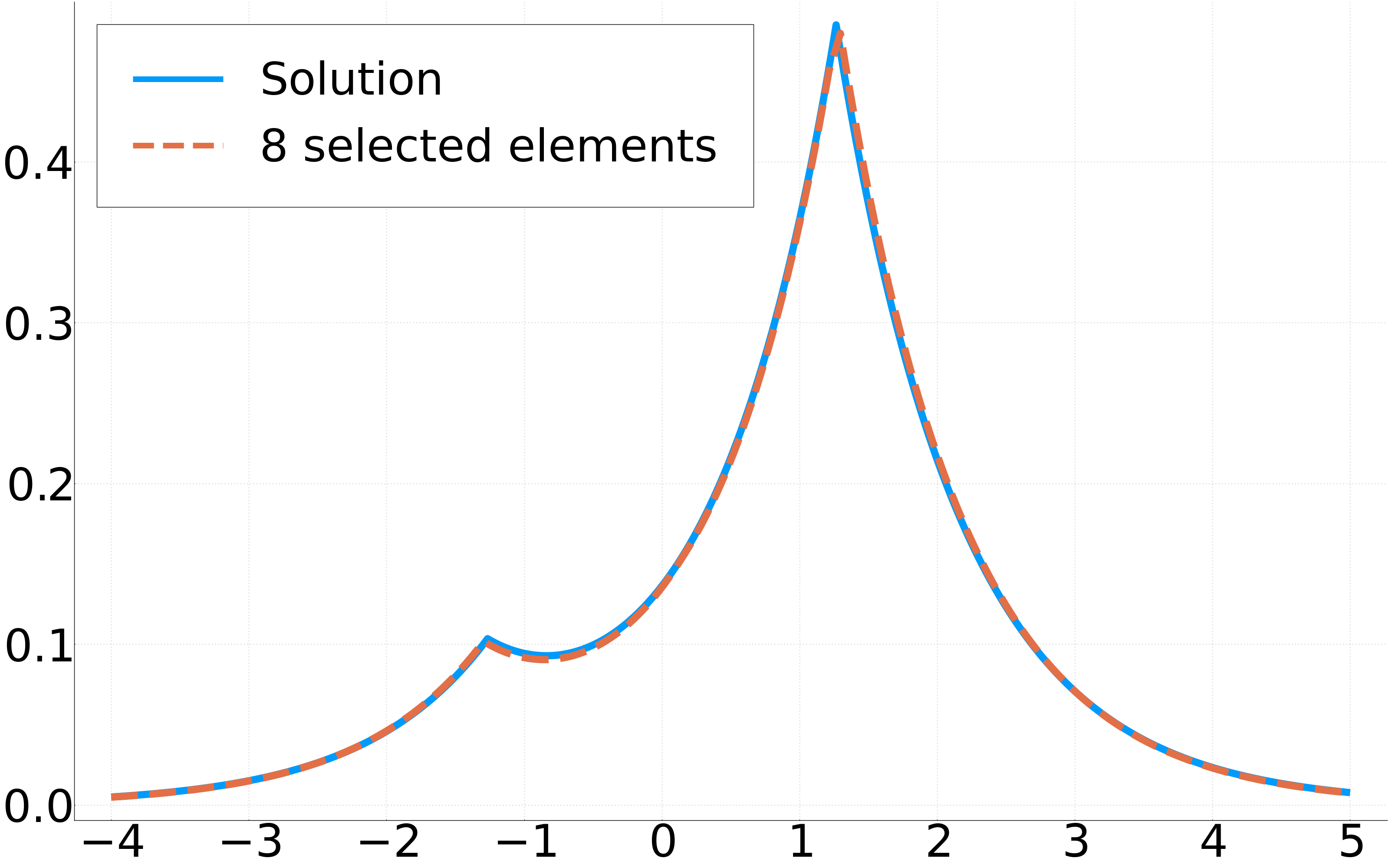}
    \vspace*{8pt}
    \caption{Example of projections on bases with 2,3,5, and 8 elements for $r = 1.266$.}
    \label{fig:proj_example}
\end{figure}

We now present the results obtained by running the offline algorithm presented in Section~\ref{sec:greedy}. We present in Figure~\ref{fig:basis7} the 7 first selected snapshots. We observe that the two first selected snapshots correspond to the extreme parameters $0.5$ and $2.5$, then the next ones are relatively well distributed across the parameter space.
In Figure~\ref{fig:proj_decay}, we plot the decrease of the projection error in $\distMW$-norm over the training set. We provide both the mean error on the training set and the maximum error. We observe that this projection error decreases very fast and seems exponentially decreasing. Moreover, we gain about two orders of magnitude between 2 and 15 added snapshots on the mean error.

In terms of computational cost, the most expensive part is the listing of the vertices of the constraint space~\eqref{eq:forthevertices}, which increases exponentially with the number of selected snapshots. However, the code can be trivially parallelized, and is indeed running on multicores. Also, in the future, a global optimization solver such as Gurobi could be used instead of the listing of the vertices, possibly loosing the global optimality of the found minimizer but gaining a lot in computational efficiency.

In Figure~\ref{fig:proj_decay_w2} we compare a greedy selection of snapshots using the 
$\distMW$-distance and the $\distW$-distance, measuring all projection errors  
in $\distW$-distance. We observe that the selection in $\distMW$-distance clearly outperforms the
$\distW$-distance-based selection, showing that the $\distMW$-distance is more adapted to the considered partial differential equation.

In Figure~\ref{fig:proj_example}, we provide a few examples of projection on the reduced basis for a snapshot with parameter $r = 1.266$, which is in $\M_{\mathbf{z}}$ but not in the training set $\M_{tr}$. We observe that the projections cannot be visually distinguished from the exact solution already when the reduced basis contains only 5 elements.

\subsection{Online phase}

In this section we provide results on the online optimization algorithm.
First, recall that the energy minimization problem~\eqref{eq:minEonline} is a global optimization problem so that Algorithm~\ref{algorithm:online} may not necessarily return the global minimizer of the problem, possibly overestimating the presented error results compared to the exact ones.
In practice, the parameters of Algorithm~\ref{algorithm:online} were chosen as follows. We used a $L = 2000$ elements Sobol sequence covering the set $B_N = [-2, 2]^N\cap\Omega_N(m^1,\ldots, m^N)$ as starting points $\bm \lambda_1, \dots, \bm \lambda_L$.

In Figure~\ref{fig:energy_heatmap} we plot an example of energy landscape $\bm \lambda \mapsto \Esol\left( \appbarMW(m^1, m^2) \right)$ heatmap with $\mathbf{z} = (0.8, 1.1)$ and $\mathbf{r} = (-r, r)$ for $r = 2.15$,
where $m^1$ and $m^2$ are the first two selected snapshots in the offline phase (see Figure~\ref{fig:basis7}). The white part corresponds to the outside of the domain $\Omega_2$.
We already observe several local minima. Moreover, the energy is nonsmooth due to the absolute values appearing in the formula, see Remark~\ref{remark:smoothing} for more information on the smoothing technique.

\begin{figure}[!t]
    \centering
    \includegraphics[width=.5\textwidth]{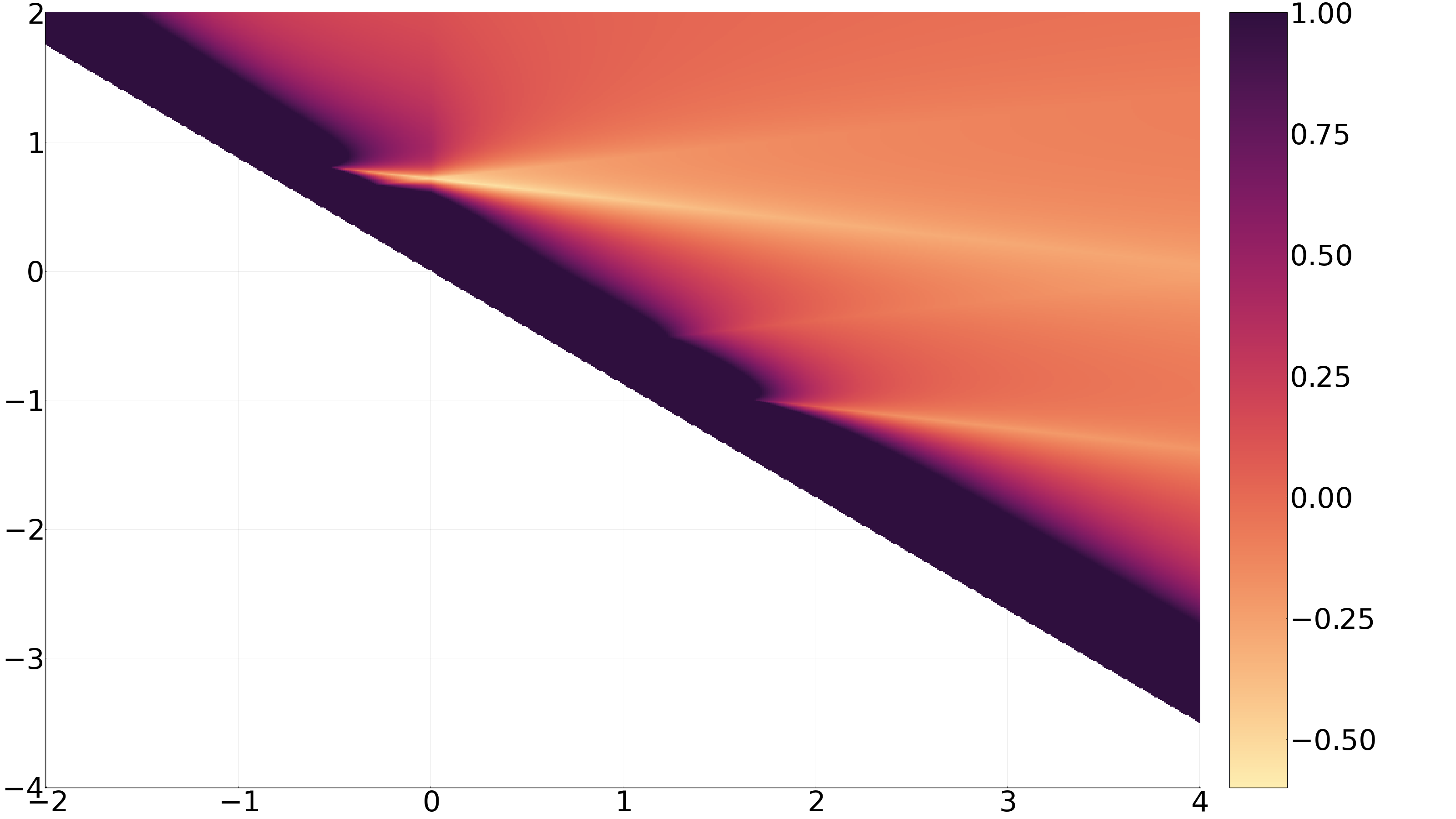}
    \vspace*{8pt}
    \caption{Heatmap of an energy functional (for $r = 2.15$) at barycenters between first two selected elements.}
    \label{fig:energy_heatmap}
\end{figure}

We now provide the plot of the error in energy as a function of the number of selected snapshots in Figure~\ref{fig:energy_decay}.
We provide both the maximum error and the mean error over a test set of $51$ equally distributed elements for $r \in [0.5, 3]$. 
We observe that the energy maximum error decreases by three orders of magnitude from 2 to 8 snapshots, which is particularly encouraging. Adding more elements in the reduced basis does not seem to improve significantly the results. 
This may  either be due to the increasing difficulty of solving the global optimization problem in larger dimension, but also to the smoothing of the energy functional that is used to avoid convergence problems.

\begin{figure}
    \centering
    \includegraphics[width=0.5\textwidth]{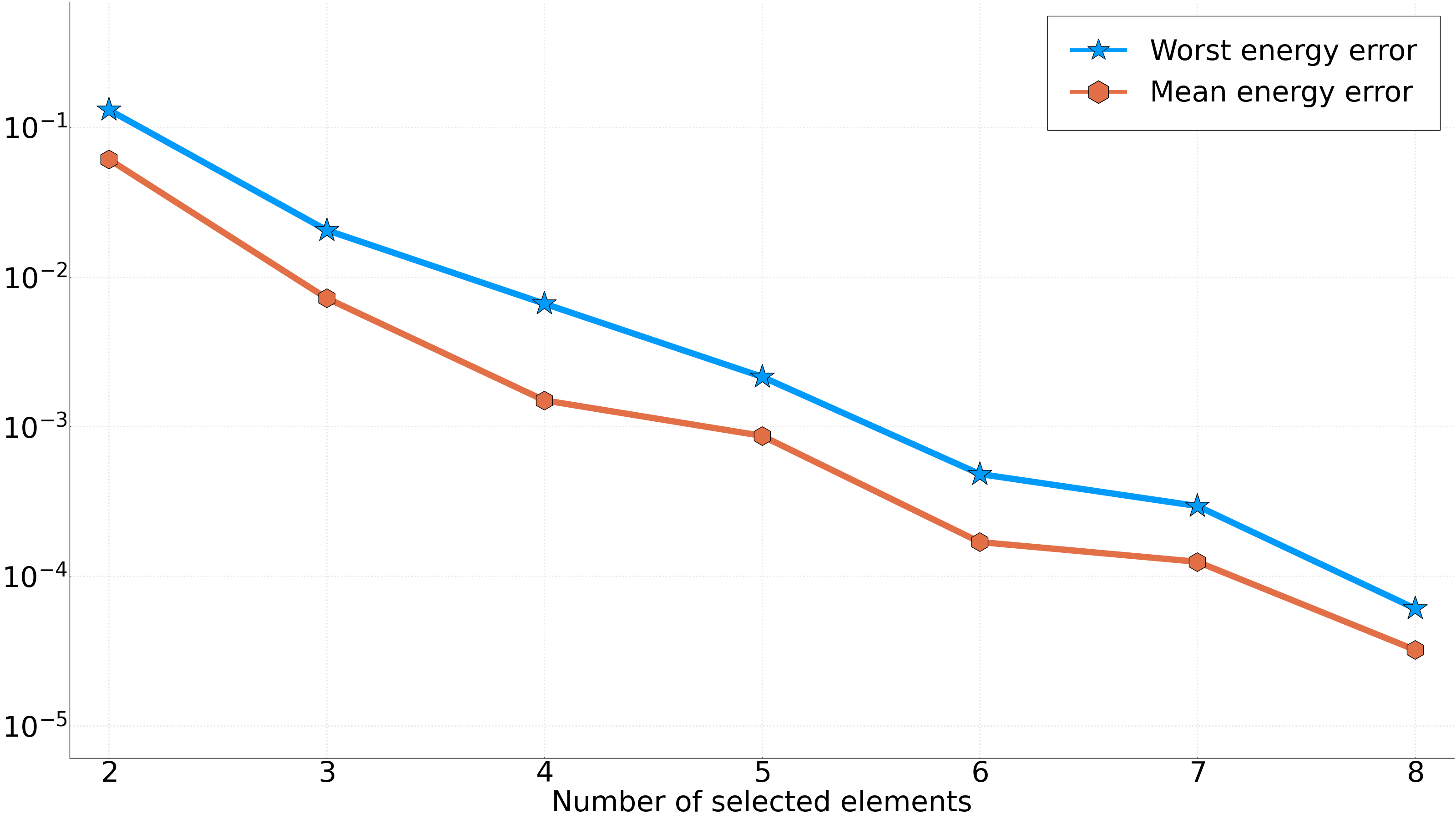}
    \vspace*{8pt}
    \caption{Decay of the energy error in the online phase for 51  equally distributed elements for $r \in [0.5, 3]$. }
    \label{fig:energy_decay}
\end{figure}

\begin{figure}[!htb]
    \centering
    \includegraphics[width=0.4\textwidth]{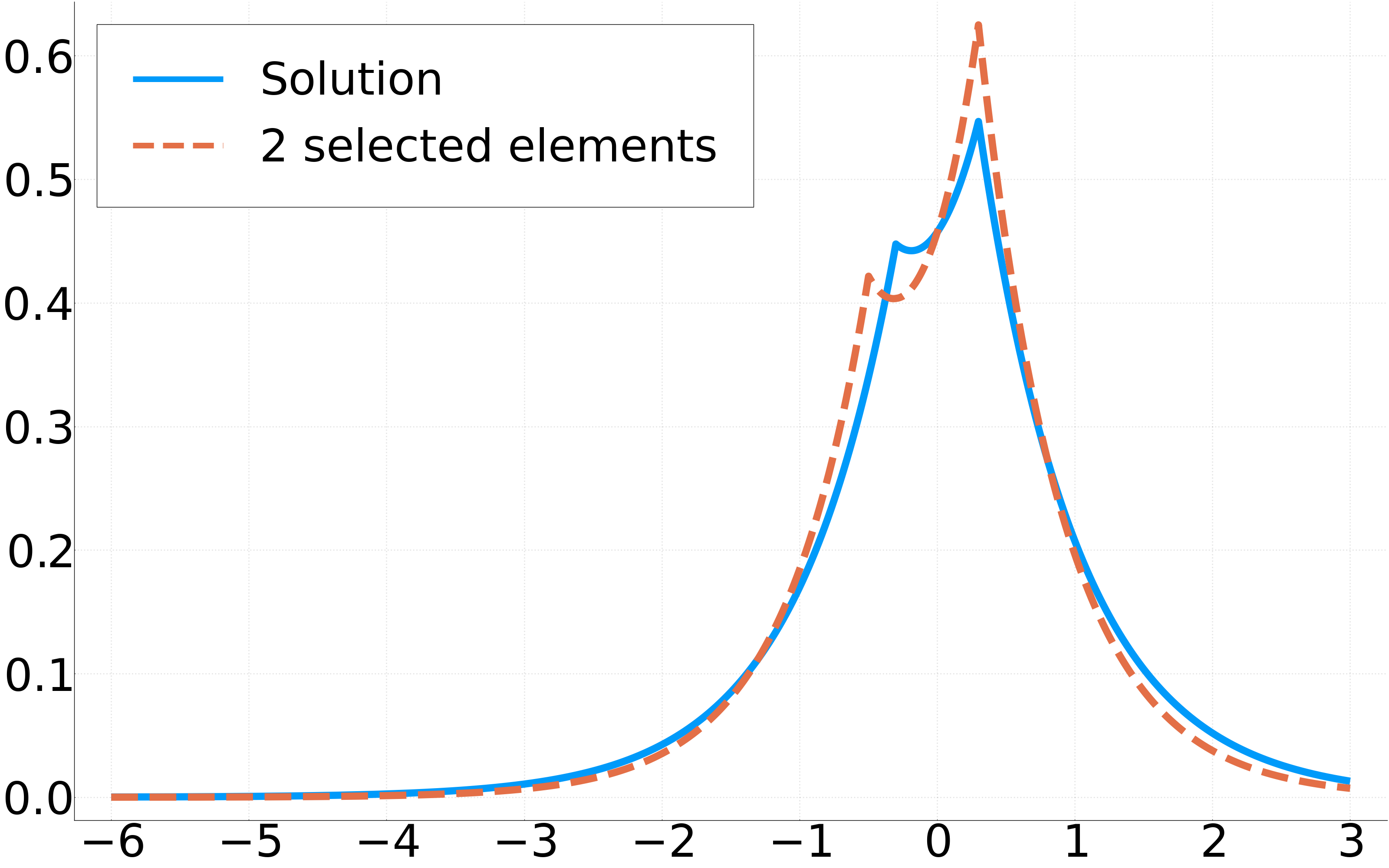}
    \includegraphics[width=0.4\textwidth]{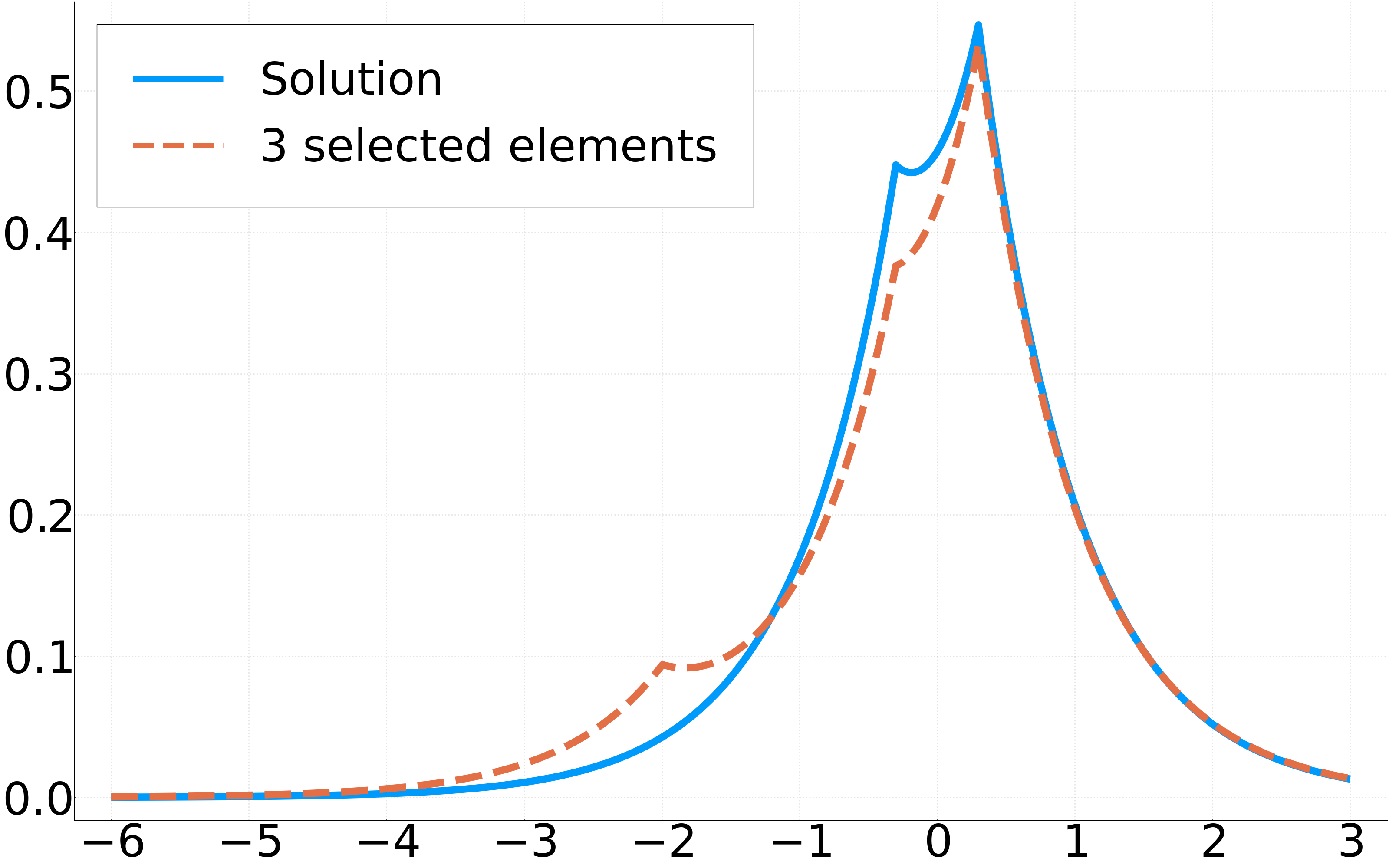}
    \includegraphics[width=0.4\textwidth]{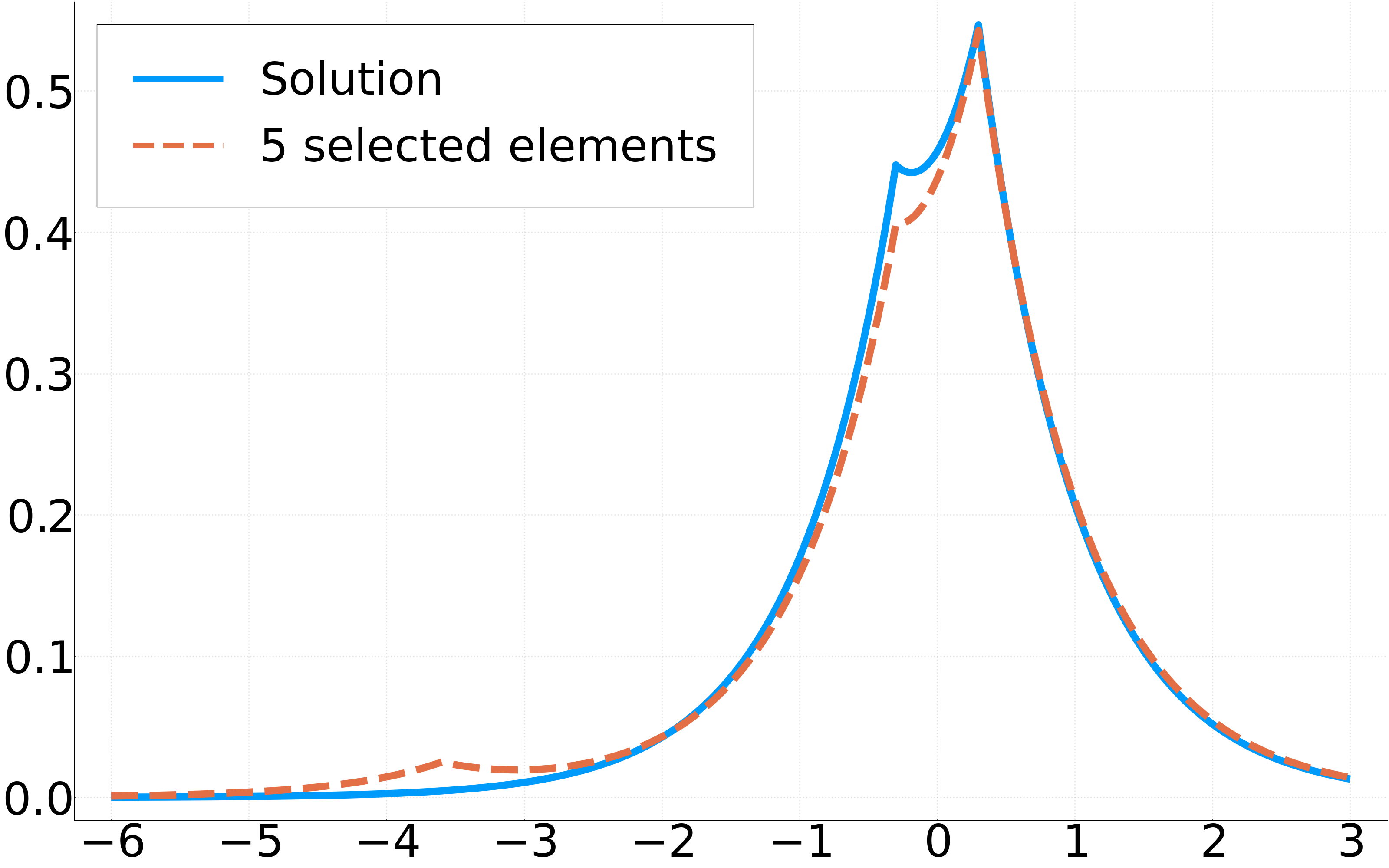}
    \includegraphics[width=0.4\textwidth]{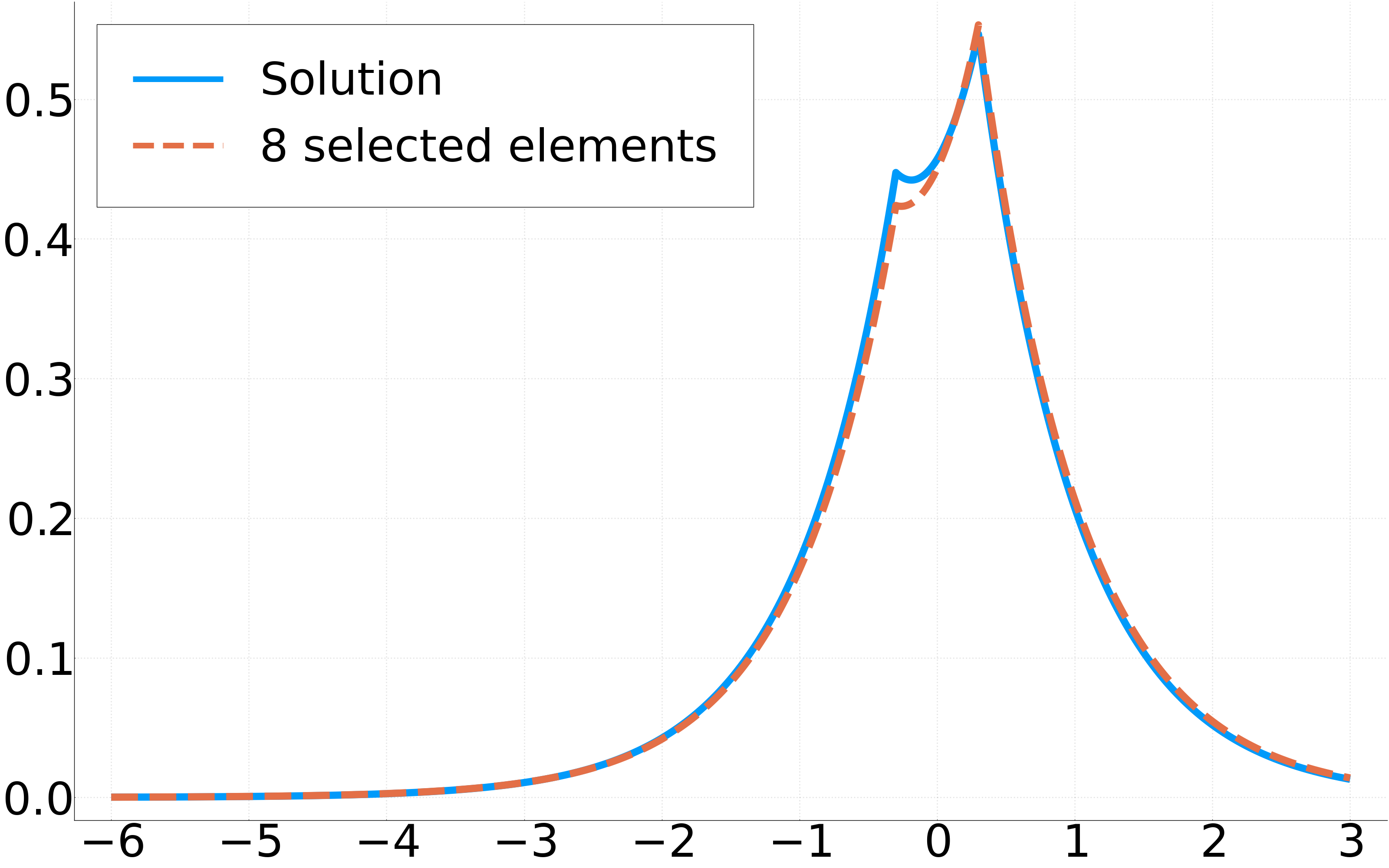}
    \vspace*{8pt}
    \caption{Extrapolation: example of energy projections for $r = 0.3$.}
    \label{fig:inter_proj_example}
\end{figure}

\begin{figure}[!htb]
    \centering
    \includegraphics[width=0.4\textwidth]{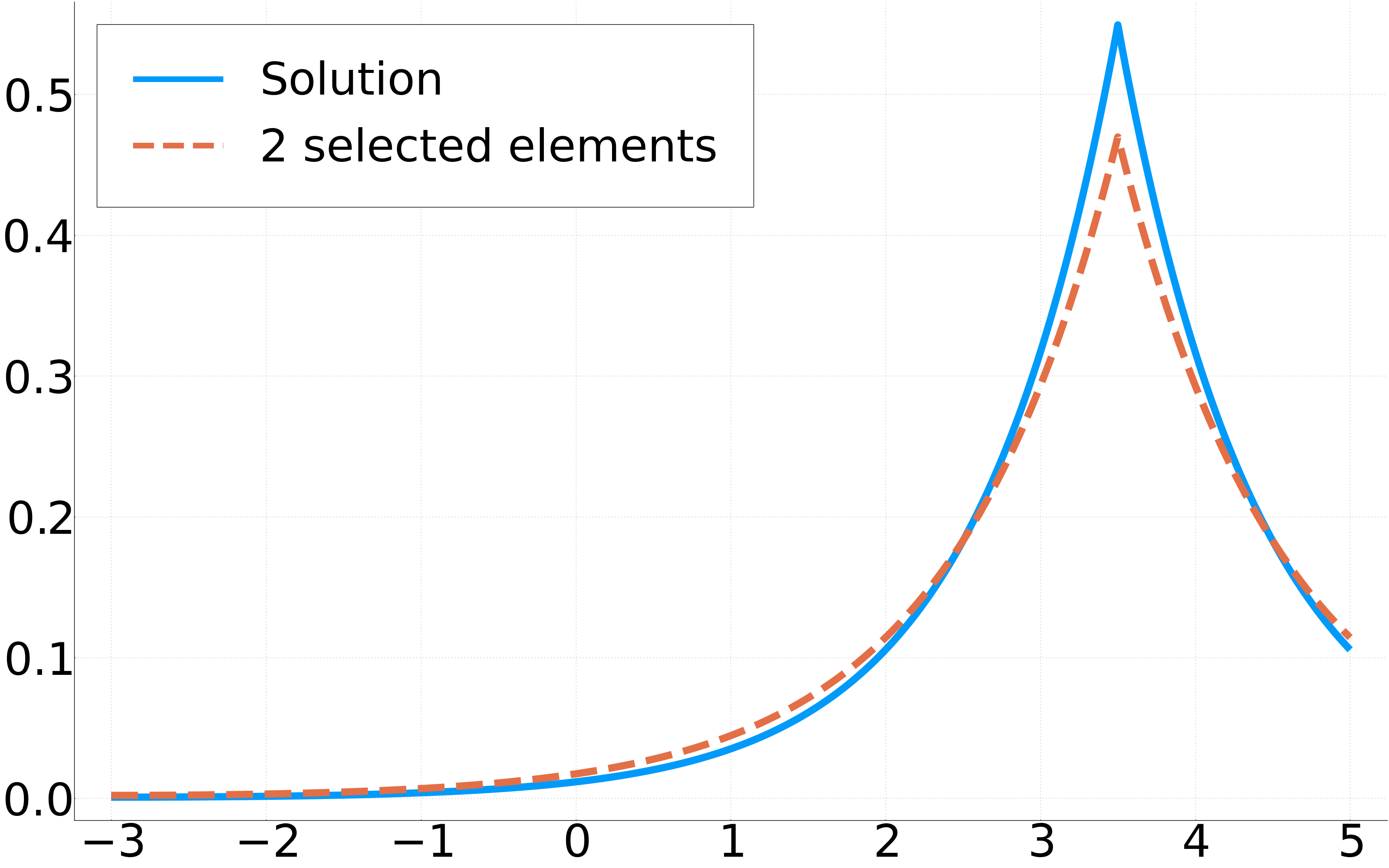}
    \includegraphics[width=0.4\textwidth]{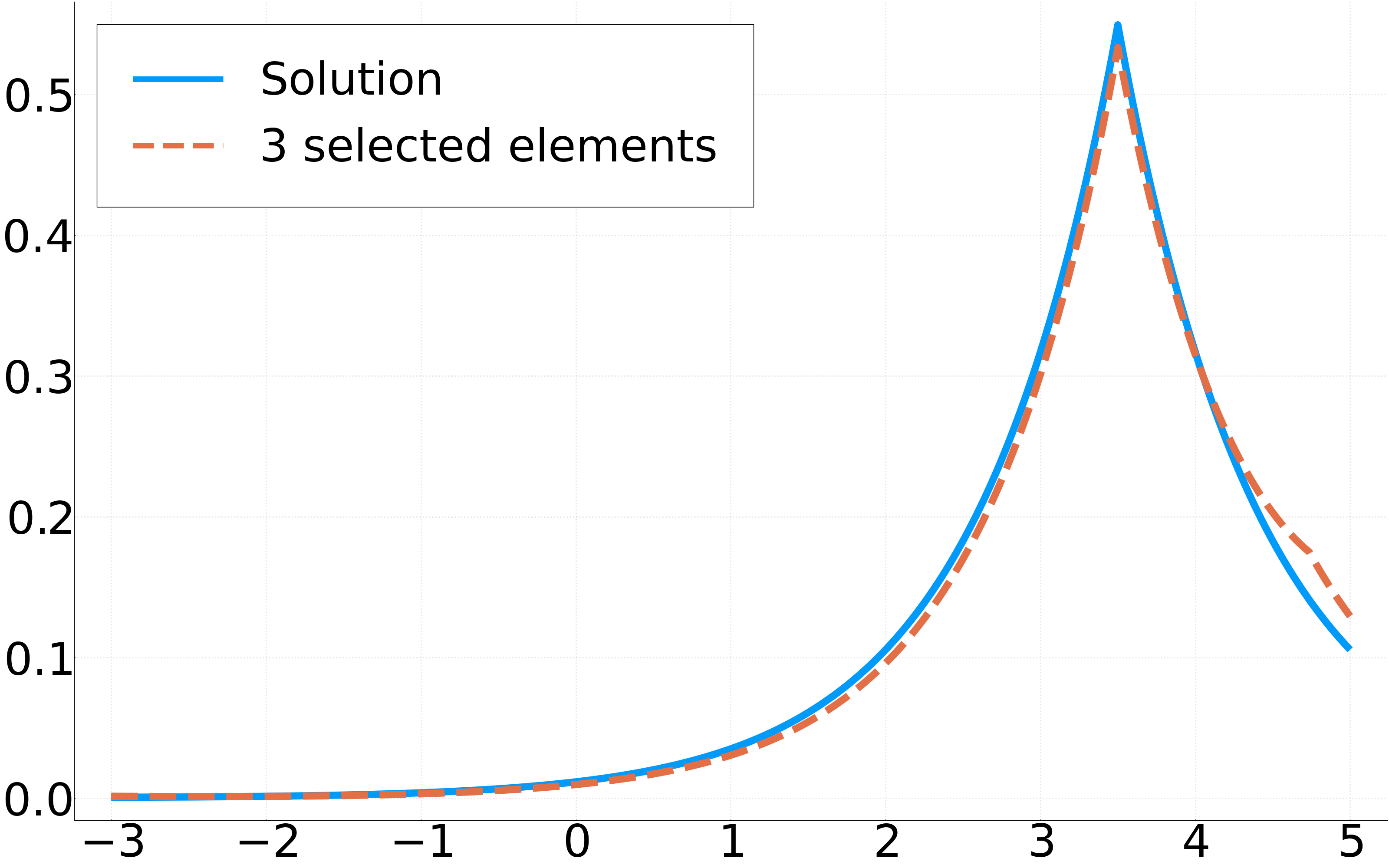}
    \includegraphics[width=0.4\textwidth]{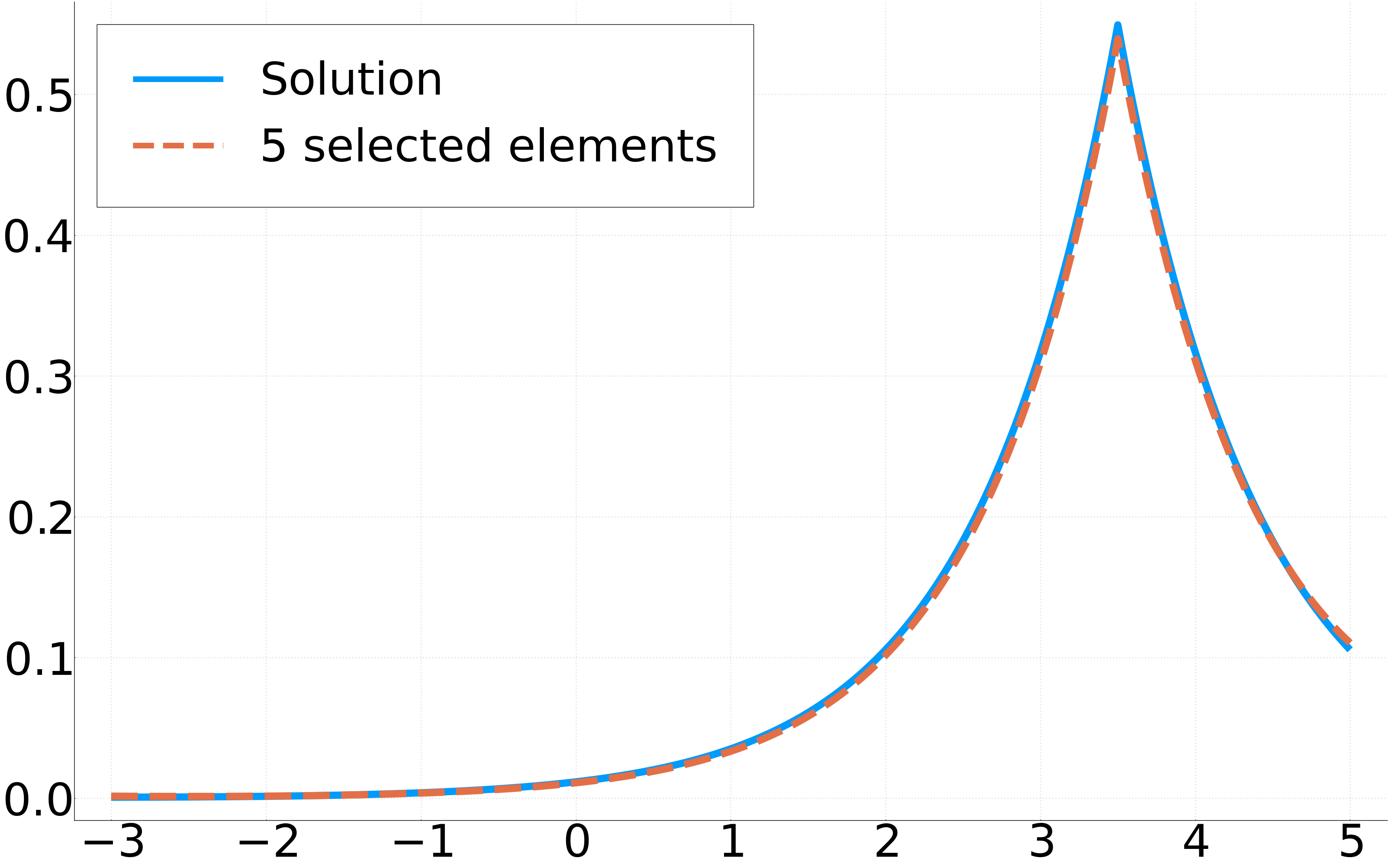}
    \includegraphics[width=0.4\textwidth]{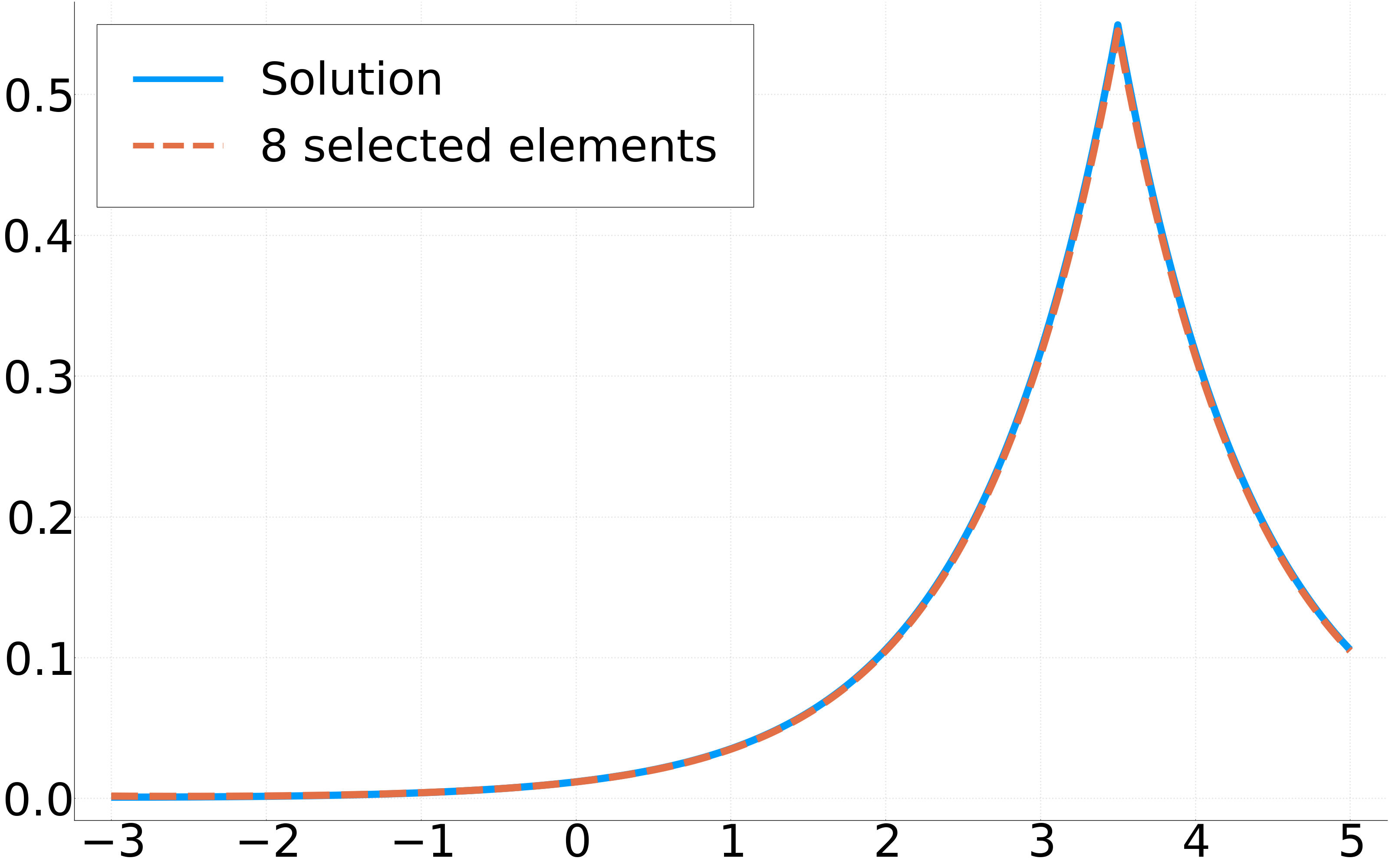}
    \vspace*{8pt}
    \caption{Extrapolation: example of energy projections for $r = 3.5$.}
    \label{fig:exter_proj_example}
\end{figure}

\begin{figure}[!htb]
    \centering
    \includegraphics[width=.5\textwidth]{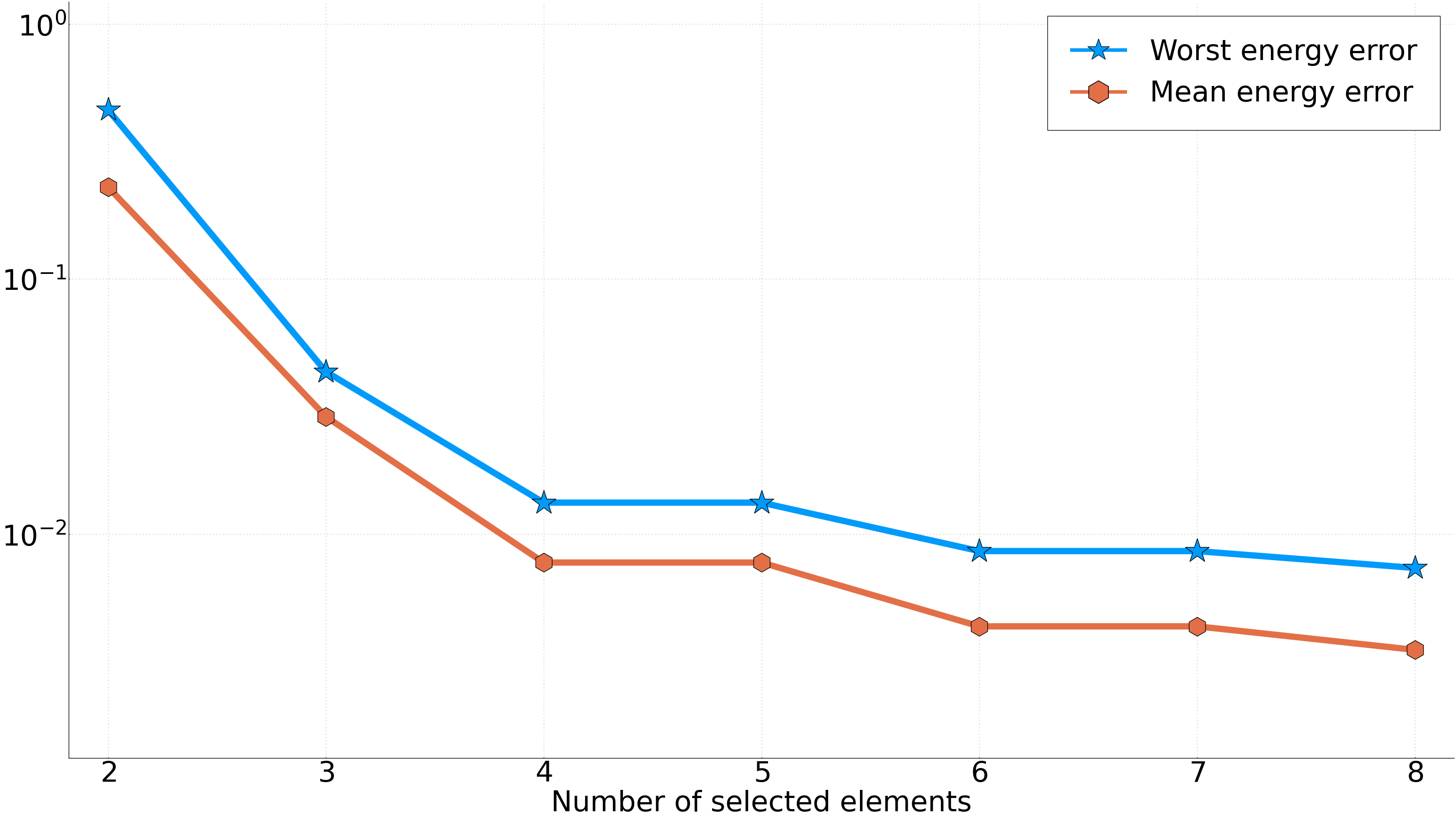}
    \vspace*{8pt}
    \caption{Extrapolation: decay of the energy error over a set of $17$ equally distributed elements for $r = 0, \dots, 0.48$.}
    \label{fig:inter_decay}
\end{figure}

\begin{figure}[!htb]
    \centering
    \includegraphics[width=.5\textwidth]{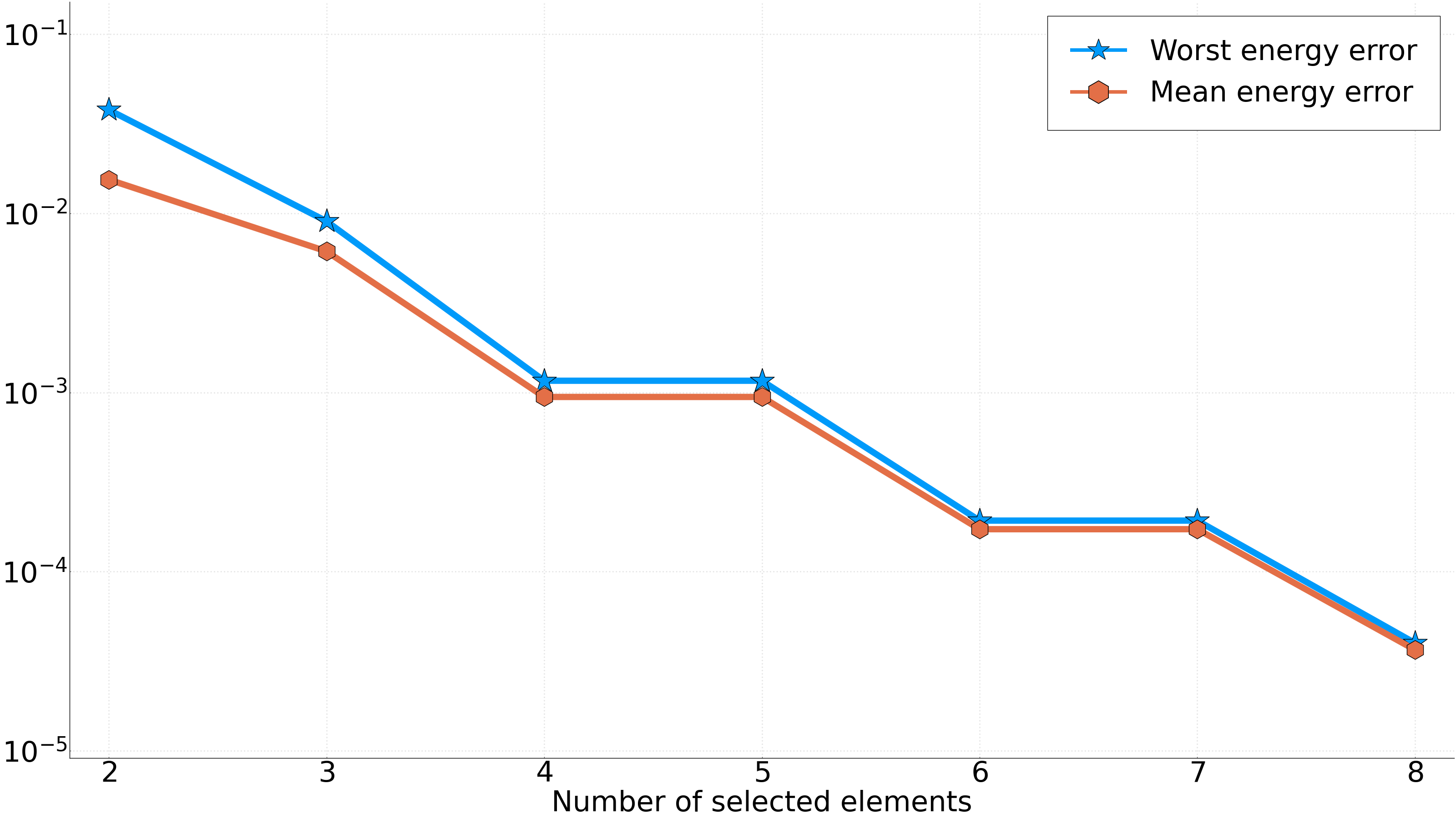}
    \vspace*{8pt}
    \caption{Extrapolation example: decay of the energy error over a set of $21$ equally distributed elements for $r\in [3.05, \dots, 4]$.}
    \label{fig:exter_decay}
\end{figure}

Finally, we provide extrapolation examples. 
On Figures.~\ref{fig:inter_proj_example} and~\ref{fig:exter_proj_example}, we show the projection of the solution on the reduced basis for 2,3,5, and 8 snapshots. We observe that 8 snapshots seems sufficient to obtain a satisfactory barycenter projection of the exact solution onto the reduced basis.
More generally, we plot on Figure~\ref{fig:inter_decay} the decay of the online error on  $17$ equally distributed elements ranging from $r = 0$ to $r= 0.48$. On Figure~\ref{fig:exter_decay}, we plot the energy online error on $21$ equally distributed elements with $r = 3.05, \dots, 4$. We observe that we obtain very accurate results with only a few snapshots in the reduced basis, although the solutions are not in the parameter range of the training set, showing the nice extrapolation capabilities of the method.

\section{Proofs}\label{sec:proofs}

The aim of this section is to gather the proofs of the theoretical results stated in Section~\ref{sec:kolmogorov}.

\subsection{Preliminary lemma}

Before going into the statement of the different theorems, we provide the following basic lemma, which bounds the error between a function with a lack of regularity at a few points and its best piecewise polynomial approximation.
It will be used to obtain upper bound of Kolmogorov $n$-widths in Theorems~\ref{thrm_linear_comp} and~\ref{thrm:transp_before}.
\begin{lemma} \label{lm:proj_error}
    Let $f$ be a real-valued function defined over a compact interval $I = [a, b]$,
    and let $a = x_0 < \dots < x_n = b$ be a mesh of maximal size $h$ on $I$.
    Suppose moreover that $f$ is of class $\C^{p+1}$ on intervals
    $I_k = [x_k, x_{k+1}]$ except a few, named $I_{k_1}, \dots I_{k_q}$ where it is of class
    $\C^{p-1}$ with $f^{(p-1)}$ absolutely continuous. Then, defining
    \begin{equation*}
        V_{n,p} = \vect\{x \mapsto x^i\mathds{1}_{I_k}(x)\}_{\substack{i = 0, \dots, p
            \\ k = 0, \dots, n-1}},
    \end{equation*}
    a vector space of dimension $n(p+1)$, we have the following projection error
    \begin{align*}
        \Vert f - \dP_{V_{n,p}}f \Vert_{\L^2(I)} & \leqslant \left( \sqrt{b-a}
            \frac{\Vert f^{(p+1)}\Vert_{\L^\infty(I \backslash ( I_{k_1} \cup\dots\cup I_{k_q} ) )}}{(p+1)!} \right. \\
            & \qquad\qquad\qquad\qquad + \left. \sqrt{q} \frac{\Vert f^{(p)}\Vert_{\L^\infty(I_{k_1} \cup\dots\cup I_{k_q})}}{p!} \right) h^{p+\frac{1}{2}}.
    \end{align*}
\end{lemma}

\begin{proof}
    We denote by $f_{n,p}$ the element of $V_{n,p}$ obtained as the $p$-th order
    Taylor polynomial at $x_k$ on each interval $I_k$ for $k\notin \{k_1, \dots, k_q\}$, i.e.
    \begin{equation} \label{eq:proof_lemma_taylor_exp}
        \forall x \in I_k, ~ f_{n,p}(x) = \sum_{i = 0}^p \frac{f^{(i)}(x_k)}{i!} (x - x_k)^i,
    \end{equation}
    and as the
    $(p-1)$-th Taylor polynomial at $x_k$ on $I_k$ for the others intervals $I_k$, which is the same
    as in~\eqref{eq:proof_lemma_taylor_exp} but with a sum up to $p -1$. On intervals $I_k$, for $k\notin \{k_1, \dots, k_q\}$,
    the remainder can be written in the Lagrange form
    \begin{equation*}
        f(x) - f_{n,p}(x) = \frac{f^{(p+1)}(\xi)}{(p+1)!}(x - x_k)^{p+1},
    \end{equation*}
    where $\xi \in I_k$. This yields
    \begin{equation} \label{eq:proof_lm:notk0}
        \Vert f - f_{n,p} \Vert_{\L^2(I_k)}^2 \leqslant
            h \left(\frac{\Vert f^{(p+1)} \Vert_{\L^\infty(I_k)}}{(p+1)!}\right)^2 h^{2(p+1)}.
    \end{equation}
    On the remaining intervals $I_k$ for $k \in \{k_1, \dots, k_q\}$, as $f^{(p-1)}$ is absolutely continuous,
    $f^{(p)}$ is defined almost everywhere and we can use the integral form of the remainder and write
    \begin{equation*}
        f(x) - f_{n,p}(x) = \int_{x_k}^x
            \frac{f^{(p)}(t)}{(p-1)!}(x - t)^{p-1} \dd[t].
    \end{equation*}
    It then follows that for all $k \in \{k_1, \dots, k_q\}$
    \begin{equation} \label{eq:proof_lm:k0}
        \Vert f - f_{n,p}\Vert_{\L^2(I_{k})}^2 \leqslant
            h \left(\frac{\Vert f^{(p)} \Vert_{\L^\infty(I_k)}}{p!}\right)^2 h^{2p}.
    \end{equation}
    Combining~\eqref{eq:proof_lm:notk0} and~\eqref{eq:proof_lm:k0}, we have
    \begin{align*}
        \Vert f - f_{n,p} \Vert_{\L^2(I)}^2
        & \leqslant \sum_{k\notin \{k_1, \dots, k_q\}} h \left(\frac{\Vert f^{(p+1)} \Vert_{\L^\infty(I_k)}}{(p+1)!}\right)^2 h^{2(p+1)}
        + \sum_{k\in \{k_1, \dots, k_q\}} h \left(\frac{\Vert f^{(p)} \Vert_{\L^\infty(I_k)}}{p!}\right)^2 h^{2p} \\
        & \leqslant \left( (b-a) \left(\frac{\Vert f^{(p+1)} \Vert_{\L^\infty(I\backslash (I_{k_1} \cup\dots\cup I_{k_q}))}}{(p+1)!}\right)^2
         +  q \left(\frac{\Vert f^{(p)} \Vert_{\L^\infty(I_{k_1} \cup\dots\cup I_{k_q})}}{p!} \right)^2 \right) h^{2p+1},
    \end{align*}
    from which we easily obtain the result.
\end{proof}

\subsection{Proof of Theorem~\ref{thrm_linear_comp}}\label{sec:proof_thrm_linear_comp}

\begin{proof}[Proof of Theorem~\ref{thrm_linear_comp}]
\bfseries Step 1: \normalfont  We first prove the lower bound of~\eqref{eq:thrm:linear_comp:deltan}.
Since $L^2(-R,R)$ can be seen as a subset of $L^2(\mathbb{R})$ (by extending functions by $0$ out of $(-R,R)$), it immediately holds that
    \begin{equation} \label{eq:prrof_linear_comp:injection}
        \ddn(\M_z^R, \L^2(\RR)) \geqslant \ddn(\M_z^R, \L^2(-R, R)).
    \end{equation}
    Let us then prove that there exists a constant $\tilde{c}_R>0$ such that 
    $$
     \tilde{c}_R n^{-\frac{3}{2}} \leqslant \ddn(\M_z^R, \L^2(-R,R)). 
    $$
We denote by $K$ the kernel
    \begin{equation*}
        K:  (-R,R)^2 \ni (x,y) \longmapsto \int_{-R}^R \msol(x) \msol(y) \dd[r],
    \end{equation*}
    and introduce the integral operator defined by
    \begin{equation*}
		T_K:
		\begin{array}{ccc}
			\L^2(-R,R) & \longrightarrow & \L^2(-R,R) \\
			\varphi & \longmapsto & \displaystyle \int_{(-R,R)}K(\cdot, y)\varphi(y)\dd[y].
		\end{array}
    \end{equation*}
The operator $T_K$ is compact since $K\in\L^2((-R,R)^2)$, self-adjoint
    because of the symmetry $K(x, y) = K(y, x)$, and non-negative. Indeed, let $f \in \L^2(-R,R)$. We have
    \begin{align*}
        \langle T_K f, f \rangle_{L^2(-R,R)}
        & = \int_{(-R,R)^2} \left( \int_{-R}^R \msol(x) \msol(y) \dd[r] \right) f(x) f(y) \dd[x]\dd[y] \\
        & = \int_{-R}^R \langle \msol, f \rangle^2_{L^2(-R,R)} \dd[r] \geqslant 0.
    \end{align*}
Thus, from the spectral theorem, there exists a Hilbert basis $(\varphi_k)_{k\in\mathbb{N}^*}$ and a non-increasing sequence of non-negative real numbers
    $(\sigma_k)_{k\in\mathbb{N}^*}$ going to $0$ as $k$ goes to $+\infty$ satisfying
    \begin{equation*}
        \forall k\in\mathbb{N}^*, ~ T_K\varphi_k = \sigma_k\varphi_k.
    \end{equation*}
    Moreover, from~\cite[(1.46)]{Cohen2015-ol}, we can link the $L^2$ Kolmogorov $n$-width to the eigenvalues via the following formula
	\begin{equation} \label{eq:proof_linear_comp:eigen_width}
		\ddn(\M_z^R, \L^2(-R,R))  = \sqrt{\sum_{k=n+1}^{+\infty}\sigma_k}.
	\end{equation}

    Then, we will use the following intermediate lemma, which is proved below.
    \begin{lemma}\label{lm:lambda}
        The spectrum of $T_K$ is equal to the set $\displaystyle \bigcup_{l\in \mathbb{N}^*} \{ \lambda_l, \mu_l\}$, where for all $l\in \mathbb{N}^*$,  
        $$
            \lambda_l = \frac{4z^4}{(2a_l^2 + z^3)^2} \quad \mbox{ and } \quad \mu_l = \frac{4 z^4}{(2 b_l^2 + z^3)^2},
        $$
        with $a_l \in \left( \frac{(l-1)\pi}{R}, \frac{(l-1)\pi}{R}+\frac{\pi}{2R} \right)$ the $l$-th positive zero
        of the function \\ $x \longmapsto x\sin(Rx) - z\cos(Rx)$ and $b_l \in \left( \frac{\pi}{2R} + \frac{(l-1)\pi}{R}, \frac{l\pi}{R} \right)$ the $l$-th positive zero of the function $x \longmapsto x\cos(Rx) + z\sin(Rx)$.
        For all $l\in \mathbb{N}^*$, let us denote by
        \begin{equation*}
            \varphi_l: x \longmapsto \cos(a_lx) \quad \mbox{ and } \quad \psi_l: x \longmapsto \sin(b_lx). 
        \end{equation*}
        Then, $\varphi_l$ (respectively $\psi_l$) is an eigenvector of $T_K$ with corresponding eigenvalue $\lambda_l$ (respectively $\mu_l$). In addition, it holds that $\{ \varphi_l, \psi_l \}_{l\in \mathbb{N}^*}$ is an orthogonal basis of $L^2(-R,R)$.
    \end{lemma}
  An immediate consequence of Lemma~\ref{lm:lambda} is that for all $l\in \mathbb{N}^*$, 
    $$
    \sigma_{2l-1} =  \lambda_l \quad \mbox{ and } \quad \sigma_{2l} = \mu_l. 
    $$
    In particular, for all $k\in \mathbb{N}^*$, $\sigma_k \geq \lambda_k$ and it thus holds that
    $$
     \ddn(\M_z^R, \L^2(-R, R))^2 = \sum_{k=n+1}^{+\infty}\sigma_k \geqslant \sum_{k=n+1}^{+\infty}\lambda_k.
     $$
 In particular, it can be easily seen that the sequence $(\lambda_k)_{k\in \mathbb{N}^*}$ is decreasing and that there exists a
    constant $\tilde{c}_R > 0$ such that $\lambda_k \geqslant \tilde{c}_R k^{-4}$. Therefore, combining~\eqref{eq:prrof_linear_comp:injection} and~\eqref{eq:proof_linear_comp:eigen_width}, we obtain that
    \begin{equation*}
 \dinfn(\M_z^R, \L^2(\RR)) = \ddn(\M_z^R, \L^2(\RR)) \geqslant \tilde{c}_R n^{-\frac{3}{2}},
    \end{equation*}
    which proves the lower bound of~\eqref{eq:thrm:linear_comp:deltan}. \\
    
    {\bfseries Step 2:} We now prove the upper bound of~\eqref{eq:thrm:linear_comp:dn}.
    For $n \in \NN^*$, let us take $r \in [-R, R]$ and $x_0 = -R < \dots < x_n = R$ the equidistant subdivision of $I = [-R, R]$.
    We also define the $2n$-dimensional vector space
    \begin{equation*}
        V_n = \vect\left\{ x \mapsto x^i\mathds{1}_{I_k}(x) \right\}_{\substack{i = 0, 1\\k = 0, \dots, n-1}},
    \end{equation*}
    where $I_k = [x_k, x_{k+1}]$, and denote by $k_r$ the index such that $r \in I_{k_r}$.
    We also define the vector space
    \begin{equation*}
        V = \vect\left\{ x \mapsto e^{zx}\mathds{1}_{(-\infty, -R)}, x \mapsto e^{-zx}\mathds{1}_{(R, +\infty)} \right\}.
    \end{equation*}
    Using these vector spaces, it is clear from the shape of the solution $\msol$ (see~\eqref{eq:mono_zeta}) that
    \begin{equation*}
        \Vert \msol - \dP_{V\oplus V_n}\msol\Vert_{\L^2(\RR)} = \Vert \msol - \dP_{V_n}\msol\Vert_{\L^2(I)}.
    \end{equation*}
    Moreover, applying Lemma~\ref{lm:proj_error} with $p = 1$ and $f = \msol$ which is twice differentiable over
    the intervals $I_k$ for $k \neq k_r$ and absolutely continuous on $I_{k_r}$, we have
    \begin{equation*}
        \Vert \msol - \dP_{V_n}\msol\Vert_{\L^2(I)} \leqslant (2R)^\frac{3}{2} \left(
            \frac{\sqrt{2R}}{2} \Vert \msol''\Vert_{\L^\infty(I \backslash I_{k_r})}
            + \Vert \msol'\Vert_{\L^\infty(I_{k_r})} \right) n^{-\frac{3}{2}}.
    \end{equation*}
    And since for $x \in \RR$,
    \begin{equation*}
        | \msol'(x) | = \frac{z^2}{2} e^{-z|x - r|} \leqslant \frac{z^2}{2} \text{~~~and~~~}
        | \msol''(x) | = \frac{z^3}{2} e^{-z|x - r|} \leqslant \frac{z^3}{2},
    \end{equation*}
    there exists a positive constant $C_R$ depending on $R$ such that
    \begin{equation*}
        \Vert \msol - \dP_{V_n}\msol\Vert_{\L^2(I)} \leqslant C_R n^{-\frac{3}{2}}.
    \end{equation*}
    We can conclude by writing that
    \begin{equation*}
        \dinfn[2n+2](\M_z^R, \L^2(\RR)) \leqslant \sup_{r\in[-R,R]}\Vert \msol - \dP_{V\oplus V_n}\msol \Vert_{\L^2(\RR)}
        \leqslant C_R n^{-\frac{3}{2}}.
    \end{equation*}
    {\bfseries Step 3:}  The two remaining bounds are easily deduced remarking that for any $n$-dimensional subspace $V_n \subset \L^2(\RR)$, we have
    \begin{equation*}
        \left( \int_{-R}^R \Vert \msol - \dP_{V_n}\msol\Vert_{\L^2(\RR)}^2 \dd[r] \right)^\frac{1}{2}
            \leqslant \sqrt{2R} \sup_{r \in [-R, R]} \Vert \msol - \dP_{V_n}\msol \Vert_{\L^2(\RR)}.
    \end{equation*}
    which concludes the proof.
\end{proof}  

We provide now the proof of Lemma~\ref{lm:lambda}.

 \begin{proof}[Proof of Lemma~\ref{lm:lambda}]
        Using the definition~ of the kernel $K$ and rearranging the integrals we remark that, for all $k,l\in \mathbb{N}^*$,
        \begin{align*}
            \langle \varphi_k, T_K \varphi_l \rangle
            &= \int_{-R}^R \left(\int_{-R}^R \msol(x)\cos(a_kx)\dd[x]\right)\left(\int_{-R}^R \msol(y)\cos(a_ly)\dd[y]\right) \dd[r] \nonumber \\
            &= \langle T_S\cos(a_k\cdot), T_S\cos(a_l\cdot)\rangle, \label{eq:TKtoTS}
        \end{align*}
        where $T_S$ is the compact and self-adjoint operator defined by 
        \begin{equation}
            T_S:
            \begin{array}{ccc}
                \L^2(-R,R) & \longrightarrow & \L^2(-R,R) \\
                \varphi & \longmapsto & \displaystyle \left( r \mapsto \int_{-R}^R \msol(x) \varphi(x)\dd[x] \right).
            \end{array}
        \end{equation}
        The self-adjointness of $T_S$ stems from the fact that $u_r(x) = u_x(r)$ for all $(x,r)\in (-R,R)$. Similarly, it holds that for all $k,l\in \mathbb{N}^*$, 
        $$
          \langle \varphi_k, T_K \psi_l \rangle  = \langle T_S\varphi_k, T_S\psi_l\rangle \quad \mbox{ and } \quad   \langle \psi_k, T_K \psi_l \rangle  = \langle T_S\psi_k, T_S\psi_l\rangle. 
        $$
        
        For $k\in \mathbb{N}^*$, we now compute $T_S\cos(a_k\cdot)$:
        \begin{equation} 
        \label{eq:TSI+-1}
            T_S\cos(a_k\cdot)(r) = \frac{z}{2} \left(\int_{-R}^r e^{z(x-r)}\cos(a_kx)\dd + \int_r^R e^{-z(x-r)}\cos(a_kx)\dd \right)
        \end{equation}
        We compute the two integrals using two integrations by parts. First
        \begin{align*}
            I^-(r) &:= \int_{-R}^r e^{z(x-r)}\cos(a_kx)\dd \\
            &= \frac{1}{a_k}\left[\sin(a_kx) e^{z(x-r)}\right]_{-R}^r - \frac{z}{a_k}\int_{-R}^r e^{z(x-r)}\sin(a_kx)\dd \\
            &= \frac{\sin(a_k r)}{a_k} + \frac{\sin(a_kR)}{a_k}e^{-z(R+r)} + \frac{z}{a_k^2}\left[\cos(a_kx) e^{z(x-r)}\right]_{-R}^r
                   - \frac{z^2}{a_k^2} I^-(r) \\
            &= \frac{\sin(a_k r)}{a_k} + \frac{z}{a_k^2}\cos(a_k r) + \frac{a_k\sin(a_kR)-z\cos(a_kR)}{a_k^2}e^{-z(R+r)} 
                   - \frac{z^2}{a_k^2} I^-(r) \\
            &= \frac{\sin(a_k r)}{a_k} + \frac{z}{a_k^2}\cos(a_k r) - \frac{z^2}{a_k^2} I^-(r),
        \end{align*}
        noting that $a_k\sin(a_kR)-z\cos(a_kR) = 0$. In the same manner we can consider the other integral $I^+(r):= \int_r^R e^{-z(x-r)}\cos(a_kx)\dd = I^-(-r)$ to find
        \begin{equation*}
            I^+(r) = -\frac{\sin(a_k r)}{a_k} + \frac{z}{a_k^2}\cos(a_k r) - \frac{z^2}{a_k^2} I^+(r).
        \end{equation*}
        Hence, continuing from~\eqref{eq:TSI+-1},
        \begin{equation*}
            T_S\cos(a_k\cdot)(r) = \frac{z}{2} \left(\frac{2z}{a_k^2}\cos(a_k r) - \frac{z^2}{a_k^2} T_S\cos(a_k\cdot)(r) \right),
        \end{equation*}
        which means that
        \begin{equation*}
            T_S\cos(a_k\cdot)(r) = \frac{2 z^2}{2 a_k^2 + z^3}\cos(a_k r).
        \end{equation*}
 Similarly, for $k\in \mathbb{N}^*$, we now compute $T_S\sin(b_k\cdot)$:
        \begin{equation} \label{eq:TSI+-}
            T_S\sin(b_k\cdot)(r) = \frac{z}{2} \left(\int_{-R}^r e^{z(x-r)}\sin(b_kx)\dd + \int_r^R e^{-z(x-r)}\sin(b_kx)\dd \right)
        \end{equation}
        We compute the two integrals using two integrations by parts. First
        \begin{align*}
            J^-(r) &:= \int_{-R}^r e^{z(x-r)}\sin(b_kx)\dd \\
            &= \frac{1}{b_k}\left[-\cos(b_kx) e^{z(x-r)}\right]_{-R}^r + \frac{z}{b_k}\int_{-R}^r e^{z(x-r)}\cos(b_kx)\dd \\
            &= -\frac{\cos(b_k r)}{b_k} + \frac{\cos(b_kR)}{b_k}e^{-z(R+r)} + \frac{z}{b_k^2}\left[\sin(b_kx) e^{z(x-r)}\right]_{-R}^r
                   - \frac{z^2}{b_k^2} J^-(r) \\
            &= - \frac{\cos(b_k r)}{b_k} + \frac{z}{b_k^2}\sin(b_k r) + \frac{b_k\cos(b_kR)+z\sin(b_kR)}{b_k^2}e^{-z(R+r)} 
                   - \frac{z^2}{b_k^2} J^-(r) \\
            &= \frac{-\cos(b_k r)}{b_k} + \frac{z}{b_k^2}\sin(b_k r) - \frac{z^2}{b_k^2} J^-(r),
        \end{align*}
        noting that $b_k\cos(b_kR)+z\sin(b_kR) = 0$. In the same manner we can consider the other integral $J^+(r):= \int_r^R e^{-z(x-r)}\sin(b_kx)\dd = -J^-(-r)$ to find
        \begin{equation*}
            J^+(r) = \frac{\cos(b_k r)}{b_k} + \frac{z}{b_k^2}\sin(b_k r) - \frac{z^2}{b_k^2} J^+(r).
        \end{equation*}
        Hence, continuing from~\eqref{eq:TSI+-}
        \begin{equation*}
            T_S\sin(b_k\cdot)(r) = \frac{z}{2} \left(\frac{2z}{b_k^2}\sin(b_k r) - \frac{z^2}{b_k^2} T_S\sin(b_k\cdot)(r) \right),
        \end{equation*}
        which means that
        \begin{equation*} 
            T_S\sin(b_k\cdot)(r) = \frac{2 z^2}{2 b_k^2 + z^3}\sin(b_k r).
        \end{equation*}

Hence, for all $k,l\in \mathbb{N}^*$, the functions $\varphi_k$ and $\psi_l$ are eigenvectors of $T_S$ with respective distinct eigenvalues
$\sigma_k = \dfrac{2 z^2}{2 a_k^2 + z^3}$ and $\tau_l = \frac{2 z^2}{2 b_k^2 + z^3}$. Thus, denoting by $\lambda_k = \sigma_k^2$ and by $\mu_k = \tau_k^2$ for all $k\in \mathbb{N}^*$, we obtain that $\displaystyle \bigcup_{k\in \mathbb{N}^*} \{\lambda_k, \mu_k\} \subset \sigma(T_K)$ where $\sigma(T_K)$ is the spectrum of $T_K$. It can also be easily checked that $\{\varphi_k, \psi_k\}_{k\in \mathbb{N}^*}$ forms an orthogonal family of functions of $L^2(-R,R)$.

\medskip

It remains to prove that 
${\rm Span}\{ \phi_k, \psi_k\}_{k\in \mathbb{N}^*}$ is dense in $L^2(-R,R)$. 
From~\cite[Theorem~2]{hammersley1953non}, it is clear that $(\psi_k)_{k\in \mathbb{N}^*}$ is an orthogonal basis of the set of odd functions of $L^2(-R,R)$. 
From~\cite{Verblunsky1954-ej}, it holds similarly that $(\phi_k)_{k\in\mathbb{N}^*}$ is an orthogonal basis of the set of even functions of $L^2(-R,R)$, hence the desired result. In particular, we then have that $\displaystyle \bigcup_{k\in \mathbb{N}^*} \{\lambda_k, \mu_k\} = \sigma(T_K)$. 
\end{proof}

\subsection{Proof of Theorem~\ref{thrm:transp_before}}\label{sec:proof_thrm_transp_before}

\begin{proof}[Proof of Theorem~\ref{thrm:transp_before}]

First note that from the definition of the solution~\eqref{eq:sol} alongside its special form in this case~\eqref{eq:sym2_zetapis},
we can explicitly compute the cumulative distribution function as well as inverse cumulative distribution function as
 \begin{equation*}
    \cdf_{\ssol}(x) =
    \begin{cases}
        \frac{1}{2} \cosh(\szeta r)e^{\szeta x}, & x < -r, \\
        \frac{1}{2} \left(1 + e^{-\szeta r}\sinh(\szeta x) \right), & -r \leqslant x < r, \\
        1 - \frac{1}{2} \cosh(\szeta r)e^{-\szeta x}, & x \leqslant -r, \\
    \end{cases}
\end{equation*}
and
\begin{equation} \label{eq:def_icdfsym}
    \icdf_{\ssol}(s) =
    \begin{cases}
        \frac{1}{\szeta} ( \ln(2s) - \ln(\cosh(\szeta r)) ) , & 0 < s < s_r, \\
        \frac{1}{\szeta} \asinh\left( e^{\szeta r} (2s-1) \right), & s_r \leqslant s < 1 - s_r, \\
        -\frac{1}{\szeta} ( \ln(2(1 - s)) - \ln(\cosh(\szeta r)) ) , & 1 - s_r < s \leqslant 1, \\
    \end{cases}
\end{equation}
with 
\begin{equation*}
    s_r = \frac{1}{2} \cosh(\szeta r) e^{-\szeta r} = \frac{1}{4}\left(1 + e^{-2\szeta r} \right) = \frac{1}{4z} \szeta,
\end{equation*}
thanks to~\eqref{eq:zeta_eq}.

    Now, for $n \in \NN^*$, we introduce $\frac{1}{4} = s_0 < \dots < s_n = \frac{3}{4}$ the equidistant
    subdivision of $I := \left[\frac{1}{4}, \frac{3}{4}\right]$, and $V$, $V_n$ the vector spaces
    respectively defined by
    \begin{equation} \label{eq:proof_transp_before:def_V}
        V = \vect \left\{
                s \mapsto \ln(2s)\mathds{1}_{\left(0, \frac{1}{4}\right)}(s),
                ~ s \mapsto \mathds{1}_{\left(0, \frac{1}{4}\right)}(s),
                 s \mapsto \ln(2(1 - s))\mathds{1}_{\left(\frac{3}{4}, 1\right)}(s),
                ~ s \mapsto \mathds{1}_{\left(\frac{3}{4}, 1\right)}(s) \right\},
    \end{equation}
    and
    \begin{equation} \label{eq:proof_transp_before:def_Vn}
        V_n = \vect\left\{s \mapsto s^i\mathds{1}_{I_k}(s)\right\}_{
            \substack{i = 0, 1, 2 \\ k = 0, \dots, n-1}},
    \end{equation}
    where $I_k$ is the interval $[s_k, s_{k+1}]$. Since

    \begin{equation} \label{eq:proof_transp_before:dn}
		\dinfn[3n+4](\T_R^-, \L^2(0, 1)) \leqslant \sup_{r\in[0,R]} 
		      \Vert \icdf_{\ssol} - \dP_{V \oplus V_n}\icdf_{\ssol} \Vert_{\L^2(0, 1)},
	\end{equation}
    we are interested in estimating the error
    $\Vert \icdf_{\ssol} - \dP_{V\oplus V_n}\icdf_{\ssol} \Vert_{\L^2(0, 1)}$ for all $r\in [0,R]$.

    First, for $s \in \left(0, \frac{1}{4}\right) \cup \left(\frac{3}{4}, 1\right)$,
    it is clear that for all $r\in [0,R]$, the error
    $| \icdf_{\ssol}(s) - \dP_{V \oplus V_n}\icdf_{\ssol}(s) |$
    is equal to $0$ from the definition of $V$.

    On the remaining interval $I$, we use Lemma~\ref{lm:proj_error} with $p = 2$ and $f = \icdf_{\ssol}$ which is three times differentiable
    on the intervals $I_k$ for $k \notin \{k_r$, $\tilde{k}_r \}$ where $k_r$, $\tilde{k}_r$ are the indices such that $s_r \in I_{k_r}$ and
    $1 - s_r \in I_{\tilde{k}_r}$, and $f'$ is absolutely continuous on the intervals $I_{k_r}$ and $I_{\tilde{k}_r}$. We have
    
    \begin{equation} \label{eq:proof_transp_before:lm_eq}
        \Vert \icdf_{\ssol} - \dP_{V\oplus V_n}\icdf_{\ssol} \Vert_{\L^2(I)} \leqslant \left(
            \frac{1}{6\sqrt{2}} \Vert \icdf_{\ssol}^{(3)}\Vert_{\L^\infty(I \backslash (I_{k_r} \cup I_{\tilde{k}_r}))}
            \frac{\sqrt{2}}{2}  \Vert \icdf_{\ssol}^{(2)}\Vert_{\L^\infty(I_{k_r} \cup I_{\tilde{k}_r})}
            \right) (2n)^{-\frac{5}{2}}.
    \end{equation}

    From~\eqref{eq:def_icdfsym}, we have
    \begin{equation*}
        \icdf_{\ssol}^{(2)}(s) =
        \begin{cases}
            - \frac{1}{\szeta s^2} , & \frac{1}{4} \leqslant s < s_r, \\
            - \frac{4}{\szeta} \frac{2s-1}{\left( e^{-2\szeta r} + (2s-1)^2\right)^\frac{3}{2}},
                & s_r \leqslant s \leqslant \frac{1}{2}, \\
        \end{cases}
    \end{equation*}
    and
    \begin{equation*}
        \icdf_{\ssol}^{(3)}(s) =
        \begin{cases}
            \frac{2}{\szeta s^3} , & \frac{1}{4} \leqslant s < s_r, \\
            \frac{8}{\szeta} \frac{2(2s-1)^2 - e^{-2\szeta r}}{\left( e^{-2\szeta r} + (2s-1)^2\right)^\frac{5}{2}},
                & s_r \leqslant s \leqslant \frac{1}{2}. \\
        \end{cases}
    \end{equation*}

    Given these formulas, we can easily bound $\icdf_{\ssol}^{(2)}$ and $\icdf_{\ssol}^{(3)}$ on $I$.
    On $\left[ \frac{1}{4}, s_r \right]$, we have
    \begin{equation*}
        |\icdf_{\ssol}^{(2)}(s)| = \frac{1}{\szeta s^2} \leqslant \frac{16}{\szeta} \leqslant \frac{16}{z}, \; \; \text{ and } \; \;
        |\icdf_{\ssol}^{(3)}(s)| = \frac{2}{\szeta s^3} \leqslant \frac{128}{\szeta} \leqslant \frac{128}{z},
    \end{equation*}
    since $\szeta \geqslant z$ by~\eqref{eq:zeta_eq}.
    On $\left[ s_r, \frac{1}{2}\right]$, we have
    \begin{equation*}
        |\icdf_{\ssol}^{(2)}(s)| = \frac{4}{\szeta} \frac{1-2s}{\left((2s-1)^2 + e^{-2\szeta r}\right)^\frac{3}{2}}
            \leqslant  \frac{4}{\szeta} e^{3\szeta r} (1-2s)
            \leqslant  \frac{2}{\szeta} e^{3\szeta r}
            \leqslant  \frac{2}{z} e^{3\zeta_R R},
    \end{equation*}
    and
    \begin{equation*}
        |\icdf_{\ssol}^{(3)}(s)| = \frac{8}{\szeta} \frac{\left|2(2s-1)^2 - e^{-2\szeta r}\right|}{
            \left( (2s-1)^2 +  e^{-2\szeta r}\right)^\frac{5}{2}}
        \leqslant  \frac{8}{\szeta} e^{5\szeta r}
        \leqslant  \frac{8}{z} e^{5\zeta_R R}
    \end{equation*}
    since 
    \begin{equation*}
        \left|2(2s-1)^2 - e^{-2\szeta r}\right|
        \leqslant\max\left( e^{-2\szeta r}, 2(2s_r-1)^2 - e^{-2\szeta r} \right)
        \leqslant 1.
    \end{equation*}
    The estimate for the remaining part $\left[\frac{1}{2}, \frac{3}{4}\right]$ follows noting that $\icdf(s-1/2) = - \icdf(s+1/2)$. It then follows from~\eqref{eq:proof_transp_before:lm_eq} and the above bounds for the second and third
    derivatives of $\icdf_{\ssol}$ that
    \begin{equation} \label{eq:proof_transp_before:final_bound}
        \Vert \icdf_{\ssol} - \dP_{V\oplus V_n}\icdf_{\ssol} \Vert_{\L^2(0, 1)} \leqslant
            \frac{1}{4z} \left( \frac{2}{3} e^{5\zeta_R R}  + e^{3\zeta_R R} \right)
            n^{-\frac{5}{2}} \leqslant C e^{5\zeta_R R} n^{-\frac{5}{2}},
    \end{equation}
    for $e^{3\zeta_R R} \geqslant 8$ and $e^{5\zeta_R R} \geqslant 16$,
    with $C> 0$ a real constant independent of $R$.
    We conclude the proof by combining~\eqref{eq:proof_transp_before:dn} and~\eqref{eq:proof_transp_before:final_bound}.
\end{proof}

\subsection{Proof of Theorem~\ref{thrm:transp_after}}\label{sec:proof_thrm_transp_after}

\begin{proof}[Proof of Theorem~\ref{thrm:transp_after}]
Let $R>0$. For any $r \geqslant R$, we define the functions $g_r$ and $h_r$ as
\begin{equation*}
    g_r = \frac{1}{2} \cdf_{\S{\szeta}{-r}}, \hspace{2cm} h_r = \frac{1}{2} \cdf_{\S{\szeta}{r}},
\end{equation*}
the cumulative distribution functions of the two Slater distributions of $u_r$, which clearly yields
\begin{equation*}
    \cdf_{\ssol} = g_r + h_r.
\end{equation*}
Moreover, we have
\begin{equation*}
    g_r(x) =
    \begin{cases}
        \frac{1}{4} e^{\szeta (x + r)}, & x < -r, \\
        \frac{1}{2} \left(1 - \frac{1}{2} e^{-\szeta (x + r)} \right), & x \geqslant -r,
    \end{cases}
    \quad h_r(x) =
    \begin{cases}
        \frac{1}{4} e^{\szeta (x - r)}, & x < -r, \\
        \frac{1}{2} \left(1 - \frac{1}{2} e^{-\szeta (x - r)} \right), & x \geqslant -r,
    \end{cases}
\end{equation*}
and
\begin{equation*}
    g_r^{-1}(s) =
    \begin{cases}
        \frac{1}{\szeta}\ln(4s) - r, & 
        \hspace{-3mm} s \in \left(0, \frac{1}{4} \right), \\
        - \frac{1}{\szeta}\ln(2(1-2s)) - r, & 
        \hspace{-3mm} s \in \left[\frac{1}{4}, \frac{1}{2} \right),
    \end{cases}
    \quad h_r^{-1}(s) =
    \begin{cases}
        \frac{1}{\szeta}\ln(4s) + r, & 
        \hspace{-3mm} s \in \left(0, \frac{1}{4} \right), \\
        - \frac{1}{\szeta}\ln(2(1-2s)) + r, & 
        \hspace{-3mm} s \in \left[\frac{1}{4}, \frac{1}{2} \right).
    \end{cases}
\end{equation*}

Consider now the vector space $V_R$
    \begin{equation} \label{eq:def_vec_VR}
        V_R= \vect\left\{ \mathds{1}_{\left(0, \frac{1}{2} \right)}, g_R^{-1},
        \mathds{1}_{\left(\frac{1}{2}, 1 \right)}, \left(\frac{1}{2} + h_R\right)^{-1} \right\}.
    \end{equation}
Our aim is to prove that there exists a positive constant $C$ independent of $r$ and $R$ such that for all $r \geqslant R$, the following bound for the projection error holds:
	\begin{equation*}
		\Vert \icdf_{\ssol} - \dP_{V_R} \icdf_{\ssol} \Vert_{\L^2(0, 1)} \leqslant C R e^{-\frac{1}{2} \zeta_R R},
	\end{equation*}
    which will yield the desired result. It roughly states that solutions $u_r$ with a large position parameter $r$ are close to the mean of the two solutions with $M = 1$ centered in $-r$ and $r$ with a fixed charge $z$.

    By parity with respect to $s=1/2$ of $\icdf_{\ssol}$ and remarking
    that $V_R$ contains functions on $\left(0, \frac{1}{2}\right)$ and their exact translations
    on $\left(\frac{1}{2}, 1\right)$, we only need to focus on the left interval $\left(0, \frac{1}{2}\right)$.
    Now, as $g_r^{-1}$ is in $V_R$, we bound the projection by
    \begin{align}
		\Vert \icdf_{\ssol} - \dP_{V_R} \icdf_{\ssol} \Vert_{\L^2(0, \frac{1}{2})}^2
            & \leqslant \Vert \icdf_{\ssol} - g_r^{-1} \Vert_{\L^2(0, \frac{1}{2})}^2 \nonumber \\
            & = \int_0^\frac{1}{2} \left| \icdf_{\ssol}(s) - g_r^{-1}(s) \right|^2 \dd[s]. \label{eq:proof_transp_after:int_err}
	\end{align}
    Next, we consider separately the integrals on $(0, g_r(0)]$ and $\left(g_r(0), \frac{1}{2} \right]$.
    To estimate the first part, we remark that $g_r^{-1} - \icdf_{\ssol}$ is positive and increasing on
    the interval $\left(0, \frac{1}{2} \right]$. Hence
    \begin{equation*}
        \int_0^{g_r(0)} \left| \icdf_{\ssol}(s) - g_r^{-1}(s) \right|^2 \dd[s] \leqslant |\icdf_{\ssol}(g_r(0))|
            \int_0^{g_r(0)} \left( g_r^{-1}(s) - \icdf_{\ssol}(s) \right) \dd[s].
    \end{equation*}
    Using inverses, we have that
    \begin{equation*}
        \int_0^{g_r(0)} \left( g_r^{-1}(s) - \icdf_{\ssol}(s) \right) \dd[s] \leqslant
            \int_{-\infty}^0 \left( \cdf_{\ssol}(x) - g_r(x) \right) \dd[x]
            = \int_{-\infty}^0 h_r(x) \dd[x]
            = \frac{1}{4\szeta} e^{-\szeta r},
    \end{equation*}
    and
    \begin{equation*}
        |\icdf_{\ssol}(g_r(0))| = \frac{1}{\szeta} \asinh\left(\frac{1}{2}\right).
    \end{equation*}
    From this, it follows that
    \begin{equation*}
        \int_0^{g_r(0)} \left| \icdf_{\ssol}(s) - g_r^{-1}(s) \right|^2 \dd[s] \leqslant
            \frac{1}{4\szeta^2} \asinh\left(\frac{1}{2}\right) e^{-\szeta r}.
    \end{equation*}
    We now bound the second part:
    \begin{align*}
        \int_{g_r(0)}^\frac{1}{2} \left| \icdf_{\ssol}(s) - g_r^{-1}(s)  \right|^2 \dd[s]
            & \leqslant 2 \int_{g_r(0)}^\frac{1}{2} \left| \icdf_{\ssol}(s) + r \right|^2 \dd[s]
             + \frac{2}{\szeta^2} \int_{g_r(0)}^\frac{1}{2} \left| \ln(2(1-2s)) \right|^2 \dd[s] \\
        & \leqslant 2r^2 \left(\frac{1}{2} - g_r(0) \right) + \frac{1}{2 \szeta^2}
            \int_{0}^{e^{-\szeta r}} \left| \ln(t) \right|^2 \dd[t] \\
        & \leqslant \frac{r^2}{2} e^{-\szeta r} + \frac{1}{2 \szeta^2}( \szeta^2 r^2 + 2\szeta r + 2)
            e^{-\szeta r} \\
        & = \left( r^2 + \frac{r}{\szeta} + \frac{2}{\szeta^2} \right) e^{-\szeta r}. 
    \end{align*}
    Combining these two estimates, we bound the integral in~\eqref{eq:proof_transp_after:int_err} by
    \begin{equation*}
		\int_0^\frac{1}{2} \left| \icdf_{\ssol}(s) - g_r^{-1}(s) \right|^2 \dd[s] \leqslant C r^2 e^{-\szeta r},
	\end{equation*}
    with $C$ a positive real constant independent $r$ and $R$.
    From this, we easily deduce that
    \begin{equation*}
		\Vert \icdf_{\ssol} - \dP_{V_R} \icdf_{\ssol} \Vert_{\L^2(0, 1)}
            \leqslant C R e^{-\frac{1}{2} \zeta_R R},
	\end{equation*}
    for $R$ large enough such that for all $r \geqslant R$, we have $r^2 e^{-\szeta r} \leqslant R^2 e^{-\zeta_R R}$.
\end{proof}

\subsection{Proof of Theorem~\ref{thrm_final}}\label{sec:proof_thrm_final}

\begin{proof}[Proof of Theorem~\ref{thrm_final}]
	We consider the vector spaces $V$ and $V_n$ as respectively defined in~\eqref{eq:proof_transp_before:def_V}
    and~\eqref{eq:proof_transp_before:def_Vn} and the vector space $V_R$ as defined in~\eqref{eq:def_vec_VR}.

	First of all, fixing $\varepsilon > 0$ there exists a constant $C_\varepsilon>0$ such that
	\begin{equation*}
		\forall r \geqslant R,~ \Vert\icdf_{\ssol} - \dP_{V_R}\icdf_{\ssol}\Vert_{\L^2(0, 1)}
		  \leqslant CR e^{-\frac{z}{2} R} \leqslant C_\varepsilon e^{-(\frac{z}{2} - \varepsilon)R},
	\end{equation*}
    from Theorem~\ref{thrm:transp_after} and~\eqref{eq:zeta_eq}. Also, from~\eqref{eq:proof_transp_before:final_bound} and~\eqref{eq:zeta_eq}, we have
    \begin{equation*}
		\forall r \in [0, R],~ \Vert\icdf_{\ssol} - \dP_{V\oplus V_n}\icdf_{\ssol}\Vert_{\L^2(0, 1)}
		  \leqslant Ce^{10z R} n^{-\frac{5}{2}}.
	\end{equation*}
	Hence, 
	\begin{align*}
        \dn[3n+4](\T, \L^2(0, 1))
        & \leqslant \sup_{r\in\RR^+} \Vert \icdf_{\ssol} - \dP_{(V\oplus V_n) + V_R}\icdf_{\ssol} \Vert_{\L^2(0, 1)} \\
	    & \leqslant \max\left(\sup_{r\in [0, R]} \Vert \icdf_{\ssol} - \dP_{V\oplus V_n}\icdf_{\ssol} \Vert_{\L^2(0, 1)};
	       \sup_{r\geqslant R} \Vert \icdf_{\ssol} - \dP_{V_R}\icdf_{\ssol} \Vert_{\L^2(0, 1)} \right) \\
	    & \leqslant \max\left( Ce^{10z R} n^{-\frac{5}{2}}; C_\varepsilon e^{-(\frac{z}{2}- \varepsilon)R} \right).
	\end{align*}
	Choosing $R_n$ satisfying
	$
		Ce^{10z R_n} n^{-\frac{5}{2}} = C_\varepsilon  e^{-(\frac{z}{2}- \varepsilon)R_n}
        $, that is
        \[
        R_n = \frac{\ln \frac{C_\epsilon}{C}  + \frac{5}{2}\ln n}{\frac{21}{2}z - \varepsilon},
	\]
	we obtain the result.
\end{proof}

\subsection{Proof of Theorem~\ref{thrm_mixture}}\label{sec:proof_thrm_mixture}

\begin{proof}[Proof of Theorem~\ref{thrm_mixture}]
    Let $m^1$ and $m^2$ be two symmetric mixtures of $K = 2$ elements, with
    parameters $r^1 = 1$, $r^2 = 2$, $\zeta^1 = 1$ and $\zeta^2 = 2$,
    that is
    \begin{equation*}
        m^1 = \frac{1}{2} \left( \S{\zeta^1}{-r^1} + \S{\zeta^1}{r^1} \right)
        \text{~~~and~~~}
        m^2 = \frac{1}{2} \left( \S{\zeta^2}{-r^2} + \S{\zeta^2}{r^2} \right).
    \end{equation*}
    Let us denote by $m_1^1 :=\S{\zeta^1}{-r^1}$, $m_2^1:= \S{\zeta^1}{r^1}$, $m_1^2:=  \S{\zeta^2}{-r^2}$ and $m_2^2:= \S{\zeta^2}{r^2}$ so that $ m^1 = \frac{1}{2} \left( m_1^1 + m_2^1 \right)$ and $m^2 =  \frac{1}{2} \left( m_1^2  + m_2^2\right)$. Since
    \begin{equation*}
        \distW\left(m^1_1, m^2_1\right)^2 = \distW\left(m^1_2, m^2_2\right)^2 = (r^1 - r^2)^2 + 2\left(\frac{1}{\zeta^1} - \frac{1}{\zeta^2} \right)^2
            = \frac{3}{2},
    \end{equation*}
    and
    \begin{equation*}
        \distW\left(m^1_1, m^2_2\right)^2 = \distW\left(m^1_2, m^2_1\right)^2 = (r^1 + r^2)^2 + 2\left(\frac{1}{\zeta^1} - \frac{1}{\zeta^2} \right)^2
            = \frac{19}{2},
    \end{equation*}
    it is easy to determine the only solution $w^*\left({\bm m}; \overline{{\bm \lambda}}\right)$ of~\eqref{eq:mixture_distance} with $\overline{\boldsymbol{\lambda}} = (1/2,1/2)$ and $\boldsymbol{m} = (m^1, m^2)$.
    Indeed, in this case, the distance $\distMW(m^1, m^2)$ simplifies to
    \begin{equation*}
        \distMW(m^1, m^2) = 2 \min_{\begin{array}{c}
        w_{11}, w_{12}\in \mathbb{R}_+\\
        w_{11} + w_{12} = \frac{1}{2}\\
        \end{array}} w_{11} \distW\left(m^1_1, m^2_1\right) + w_{12} \distW\left(m^1_1, m^2_2\right)
        = \min_{\begin{array}{c}
        w_{11}, w_{12}\in \mathbb{R}_+\\
        w_{11} + w_{12} = \frac{1}{2}\\
        \end{array}} 3w_{11} + 19w_{12},
    \end{equation*}
    which is clearly attained for $w_{11} = \frac{1}{2}$ and $w_{12} = 0$. Hence, the unique solution $w^*\left({\bm m}; \overline{{\bm \lambda}}\right) = \begin{pmatrix} \frac{1}{2} & 0 \\ 0 & \frac{1}{2} \end{pmatrix}$.

    \medskip
    
    Then the barycenters between $m^1$ and $m^2$ for a weight 
    \[ 
    {\bm \lambda} = (\lambda_1,\lambda_2)\in \Omega_2(m^1,m^2):=\left\{ (\lambda_1, \lambda_2)\in \mathbb{R}^2: \; \lambda_1 + \frac{1}{2}\lambda_2 > 0\right\}
    \]
    (since for all $(\lambda_1, \lambda_2)\in \mathbb{R}^2$ , $\frac{\lambda_1}{\zeta^1} + \frac{\lambda_2}{\zeta^2} = \lambda_1 + \frac{1}{2}\lambda_2$) are equal to
    \begin{equation*}
        \appbarMW(m^1, m^2) = \frac{1}{2} \left( \S{-r^{\bm \lambda}}{\zeta^{\bm \lambda}} +
            \S{r^{\bm \lambda}}{\zeta^{\bm \lambda}} \right),
    \end{equation*} 
    where
    \begin{equation*}
        r^{\bm \lambda} = \lambda_1 r^1 + \lambda_2 r^2 = \lambda_1 + 2\lambda_2
        \text{~~~and~~~}
        \zeta^{\bm \lambda} = \left[ \frac{\lambda_1}{\zeta^1} + \frac{\lambda_2}{\zeta^2} \right]^{-1} =  \left[ \lambda_1 + \frac{1}{2}\lambda_2 \right]^{-1}.
    \end{equation*}
    Let $u_r \in \M_\charges$.
    To show that $u_r$ is indeed an approximate mixture Wasserstein barycenter of $m^1$ and $m^2$, we find $\bm \lambda \in \RR^2$ such that $r^{\bm \lambda} = r$ and $\zeta^{\bm \lambda} = \szeta$,
    which amounts to solving the linear system
    \begin{equation*}
           \begin{cases}
               \lambda_1 + 2\lambda_2 = r \\
               \lambda_1 + \frac{1}{2}\lambda_2 = \szeta^{-1},
           \end{cases}
    \end{equation*}
    the unique solution $(\lambda_1, \lambda_2)$ of which is given by    
       $\displaystyle \lambda_1 = \frac{1}{3}( -r + 4\szeta^{-1} )$ and $\displaystyle\lambda_2 = \frac{2}{3}( r - \szeta^{-1} )$) and can be easily checked to belong to $\Omega_2(m^1,m^2)$. 
\end{proof}

\section{Conclusion}

In this work, we have focused our attention on a one-dimensional parametrized toy problem, which is insightful in terms of the difficulties faced by standard model order reduction methods to accelerate the resolution of parametrized electronic structure problems. We proved that the linear Kolmogorov $n$-width of solution sets for this equation decays at a slow algebraic rate with respect to $n$. We proved that modified Kolmogorov widths, based on optimal transport tools, in particular Wasserstein mixture distances, decay much faster. Motivated by this result, we proposed a modified greedy algorithm, precisely based on mixture Wasserstein distances and corresponding barycenters, which gives highly encouraging results. Our aim is now to export the ideas and concepts of the present work in order to build efficient reduced order models for realistic parametrized electronic structure problems in a forthcoming work.

The nonlinear reduced basis strategy presented in this work can, in principle, be adapted to application domains beyond electronic structure calculations. 
In particular, one could consider atomic measures different from Slater distributions with localized support—for example, Wigner distributions—in order to construct nonlinear reduced order models for problems defined on bounded subsets
of $\mathbb{R}^d$, as is often the case in mechanical engineering applications. 
However, the  present approach has some important limitations: 
(i) Outputs of interest must be functions with constant sign (at least up to some simple transformations) to relate them to probability distributions. (ii) One cannot prescribe complex boundary conditions on the solutions of the reduced order model, which may be a bottleneck towards the extension of the present methodology to realistic engineering applications.

\section*{Acknowlegments}

The authors thank Alexandre Nou for pointing out the Paley--Wiener results linked to the proof presented in Appendix. 
They also thank Christoph Ortner for interesting discussions.

This project has received funding from the
European Research Council (ERC) under the European Union's Horizon 2020
research and innovation program (grant agreements EMC2 No 810367 and HighLEAP No 101077204). This work
was supported by the French ‘Investissements d’Avenir’ program, project Agence
Nationale de la Recherche (ISITE-BFC) (contract ANR-15-IDEX-0003), as well as the ANR project NUMERIQ (ANR-24-CE46-2255). GD was also supported by the Ecole des Ponts-ParisTech. 
GD aknowledges the support of the region Bourgogne Franche-Comt\'e.
VE acknowledges support from the ANR project COMODO (ANR-19-CE46-0002).

\bibliographystyle{siam}
\bibliography{biblio.bib}

\end{document}